\documentclass[letter,12pt]{article}

\usepackage{amssymb}%*
\usepackage{amsthm}
\usepackage{amsfonts}%*
\usepackage[latin2]{inputenc}%*
\usepackage{amsmath}%*
\usepackage{anysize}%*
\usepackage{hyperref}
\usepackage{bbm}
\marginsize{2.6cm}{2.6cm}{1cm}{2cm}

\newtheorem*{theorem*}{Theorem}
\newtheorem*{claim*}{Claim}
\newtheorem{main_theorem}{Theorem}
\newtheorem{theorem}{Theorem}[section]

\newtheorem{main_corollary}[main_theorem]{Corollary}

\newtheorem{lemma}[theorem]{Lemma}
\newtheorem{corollary}[theorem]{Corollary}
\newtheorem{remark}[theorem]{Remark}

\newtheorem{proposition}[theorem]{Proposition}

\def\del{\partial}
\def\dbar{\bar\partial}

\def\del{\partial}

\def\o{\omega}

\title{The Mabuchi Geometry of Finite Energy Classes}

\author{Tam\'as Darvas\thanks{Research supported by NSF grant DMS1162070 and BSF grant 2012236. \newline  2010 Mathematics subject classification 53C55, 32W20, 32U05. }}
\date{\vspace{-0.3in}}

\begin{document}
\maketitle
\begin{abstract}
We introduce different Finsler metrics on the space of smooth K\"ahler potentials that will induce a natural geometry on various finite energy classes $\mathcal E_{\tilde \chi}(X,\o)$. Motivated by questions raised by R. Berman, V. Guedj and Y. Rubinstein, we characterize the underlying topology of these spaces in terms of  convergence in energy and give applications of our results to existence of K\"ahler-Einstein metrics on Fano manifolds.
\end{abstract}

\section{Introduction and Main Results}

Suppose $(X,\o)$ is a compact connected K\"ahler manifold and $\mathcal H$ is the space of K\"ahler potentials associated to $\o$.
We have three major goals in this paper. First, in hopes of unifying the treatment of different geometries/topologies on $\mathcal H$ already present in the current literature, we introduce different Finsler structures on this space. As motivation we single out two well known examples. The path length metric $d_2$ associated to the $L^2$ Riemannian structure induces the much studied Mabuchi geometry of $\mathcal H$ \cite{m}. As it turns out, by introducing the analogous $L^1$ metric and the associated path length metric $d_1$, one recovers the extremely useful strong topology introduced in \cite{bbegz}, whose study lead to many important applications in the study of weak solutions to complex Monge--Amp\'ere equations (see \cite{bbegz,bbgz,egz}). In this part we point that one can treat these different structures in a unified manner, by working with certain very general Orlicz--Finsler type metrics on $\mathcal H$.

Second, with applications in mind, one would like to characterize the metric geometry of these Orlicz--Finsler structures using very concrete terms. Until now, the lack of tools to give a satisfactory answer to this problem presented a formidable roadblock in drawing geometric conclusions using the $L^2$ geometry of $\mathcal H$. With this in mind V. Guedj \cite[Section 4.3]{g} conjectured the following pluripotential theoretic characterization:
\begin{flalign*}
\frac{1}{C}\Big({\int_X (u_1 - u_0)^2 \o_{u_0}^n} + &{\int_X (u_1 - u_0)^2 \o_{u_1}^n}\Big) \leq d_2(u_0,u_1)^2 \\ & \leq
C \Big({\int_X (u_1 - u_0)^2 \o_{u_0}^n} + {\int_X (u_1 - u_0)^2 \o_{u_1}^n}\Big), \ u_0,u_1 \in \mathcal H,
\end{flalign*}
for some $C>1$.
We prove this result in the setting of the more general Orlicz-Finsler structures mentioned above. As a consequence we note that convergence in these very general spaces implies convergence in capacity of the potentials in the sense of \cite{k} and that $\sup_X u$ is always controlled by $d_2(0,u)$ for any $u \in \mathcal H$.

Third, we give some immediate applications of the tools developed here. Parallelling the analogous results for the Calabi metric in \cite{cr} and answering questions posed in this same paper, we connect the existence of Kahler--Einstein metrics on Fano manifolds with stability of the K\"ahler--Ricci flow and properness of Ding's $\mathcal F$--functional measured in these general geometries. We  defer other applications to a future correspondence.

\subsection{Finsler Metrics on the Space of K\"ahler Potentials}
Given $(X^n,\o)$, a connected compact K\"ahler manifold, the space of smooth K\"ahler potentials $\mathcal H$ is the set
$$\mathcal H = \{ u \in C^\infty(X) | \ \o_u:=\o + i\partial \bar \partial u > 0\}.$$
Clearly, $\mathcal H$ is a Fr\'{e}chet manifold as an open subset of $C^{\infty}(X)$. For $v \in \mathcal{H}$ one can identify $T_v \mathcal{H}$ with $C^{\infty}(X)$. Given a normalized Young weight $\chi:\Bbb R \to \Bbb R^+ \cup\{+\infty\}$ (convex, even, lower semi--continuous  and satisfying the normalizing conditions $\chi(0)=0$, $1 \in \partial \chi(1)$) one can introduce a  Finsler metric on $\mathcal H$ associated to the Orlicz norm of $\chi$ (see Section 2.1):
\begin{equation}\label{FinslerDef}
\|\xi\|_{\chi,u}=\inf \Big\{ r > 0 :  \frac{1}{\textup{Vol}(X)}\int_X \chi\Big(\frac{\xi}{r}\Big) \o_u^n \leq \chi(1) \Big\}, \ \ \ \ \xi \in T_v \mathcal{H}.
\end{equation}
If $\chi(l)=\chi_p(l)=l^p/p$, then $\|\xi\|_{\chi,u}$ is just the $L^p$ norm of $\xi$ with respect to the volume  $\o_u^n$, hence \eqref{FinslerDef} generalizes the much studied Mabuchi Riemannian metric, initially investigated in \cite{m,s,do}, which corresponds to the case $p=2$. As we shall see later, the path length metric associated to the case $p=1$ is tied up with the strong topology of $\mathcal H$ introduced in \cite[Section 2.1]{bbegz}.

Unfortunately, the weights $\chi_p(l)=l^p/p$ are not twice differentiable for $1 \leq p < 2$. This will cause problems as we shall see, and an approximation with smooth weights is necessary to carry out even the most basic geometric arguments. For this reason, one needs to work with a with a fairly big class of Young weights,  even if one only wants to study the $L^p$ geometry of $\mathcal H$. The classes $\mathcal W^+_p, \ p \geq 1$ are sufficient for this purpose. These weights are finite and  satisfy the growth condition:
$$l\chi'(l) \leq p \chi(l), \ l > 0.$$

Some of our results below hold for even more general weights, however the proofs would be much more complicated, and it is unlikely that greater generality will find applications in K\"ahler geometry.

As usual, the length  of a smooth curve $[0,1]\ni t \to \alpha_t \in \mathcal{H}$ is computed by the formula:
\begin{equation}\label{curve_length_def}
l_\chi(\alpha)=\int_0^1\|\dot \alpha_t\|_{\chi,\alpha_t}dt.
\end{equation}
The distance $d_\chi({u_0},{u_1})$ between ${u_0},{u_1} \in \mathcal H$ is the infimum of the length of smooth curves joining ${u_0}$ and ${u_1}$. In case of the $L^2$ metric X.X. Chen proved that $d_2(u_0,u_1)=0$ if and only if $u_0 = u_1$, thus $(\mathcal H, d_2)$ is a metric space \cite{c}. Unfortunately, smooth geodesics don't run between the points of $\mathcal H$ \cite{lv}, however Chen proved (with complements by B\l ocki \cite{bl2}) that for $u_0,u_1 \in \mathcal H$ there exists a curve
\begin{equation}\label{Intrweakgeod}
[0,1] \ni t \to u_t \in \mathcal H_\Delta = \textup{PSH}(X,\o) \cap \{ \Delta u \in L^\infty(X)\}
\end{equation}
that satisfies the geodesic equation in a certain weak sense, and it also holds that
$$d_2(u_0,u_1)=\|\dot u_t \|_{2,u_t}, \ t \in [0,1].$$

With our first theorem we extend these results to the more general Finsler geometric setting:

\begin{main_theorem}[Theorem \ref{ChiChenTheorem}] \label{XXChenThm} If $\chi \in \mathcal W^+_p, \ p \geq 1$ then $(\mathcal H,d_\chi)$ is a metric space and for any $u_0,u_1 \in \mathcal H$ the weak geodesic from \eqref{Intrweakgeod} satisfies
$$d_\chi(u_0,u_1)=\|\dot u_t \|_{\chi,u_t}, \ t \in [0,1].$$
\end{main_theorem}
The proof of this result follows Chen's original arguments and involves a careful differential analysis of Orlicz norms. Although the metric spaces $(\mathcal H, d_\chi)$ are not equivalent (they have different metric completion, see below), it is remarkable that the weak geodesic of \eqref{Intrweakgeod} is able to see all the different metrics $d_\chi$. Topping this, as we will see shortly, this same curve is an honest geodesic in the metric completion each space $(\mathcal H, d_\chi)$.

Building on \cite{d}, which considers the $L^2$ metric, we want to describe the metric completion of $(\mathcal H, d_\chi), \chi \in \mathcal W^+_p$. Before this, we need to recall a few facts about finite energy classes. Given $u \in \textup{PSH}(X,\o)$, as explained in  \cite{gz}, one can define the non-pluripolar measure $\o_u^n$ that coincides with the usual Bedford-Taylor volume when $u$ is bounded. We say that $\o_u^n$ has full volume ($u \in \mathcal E(X,\omega)$) if $\int_X \o_u^n = \int_X \o^n$. To be compatible with the notation in \cite{gz}, to all $\chi \in \mathcal W^+_p$ we attach another weight $\tilde \chi$, defined as follows:
\begin{equation}
\tilde \chi(l) = \begin{cases} -\chi(l) &\mbox{if } l \leq 0 \\
0 & \mbox{if } l >0. \end{cases}
\end{equation}
Given  $v \in \mathcal E(X,\omega)$, we say that $v \in \mathcal E_{\tilde \chi}(X,\omega)$ if
$$E_{\tilde \chi}(v)=\int_X \tilde \chi(v) \o_v^n > -\infty,$$
Following \cite{gz}, one can introduce finite energy classes for more general weights. We will focus only on the special classes introduced above which already received attention in \cite[Section 3]{gz}, where they are referred to as classes of high energy. For a quick review of all the results needed about finite energy classes in our study we refer to \cite[Section 2.3]{d}.

Next we introduce a geodesic metric space structure on $\mathcal E_{\tilde \chi}(X,\o)$. Suppose $u_0,u_1 \in \mathcal E_{\tilde\chi}(X,\omega)$.  Let $\{u^k_0\}_{k \in \Bbb N},\{u^k_1 \}_{k \in \Bbb N}\subset \mathcal H$ be sequences decreasing pointwise to $u_0$ and $u_1$ respectively. By \cite{bk} it is always possible to find such approximating sequences. We define the metric $d_\chi(u_0,u_1)$ as follows:
\begin{equation}\label{EchiDistDef}
d_\chi(u_0,u_1) = \lim_{k \to \infty} d_\chi(u^k_0,u^k_1).
\end{equation}
Clearly, one has to argue that this limit exists and is independent of the approximating sequences, and we do this in Section 4 below. Let us also define geodesics in this space. Let $u^k_t : [0,1] \to \mathcal H_{\Delta} := \textup{PSH}(X,\omega) \cap \{ \Delta u \in L^\infty\}$ be the weak geodesic joining $u^k_0,u^k_1$, whose existence was proved by Chen (see Section 2.2). We define $t\to u_t$ as the decreasing limit:
\begin{equation}\label{EchiGeodDef}
u_t = \lim_{k \to + \infty}u^k_t, \ t \in (0,1).
\end{equation}
The curve $t \to u_t$ is well defined and we additionally have $u_t \in \mathcal E_{\tilde\chi}(X,\omega), \ t \in (0,1)$, as follows from  \cite[Theorem 6]{d}.

\begin{main_theorem}[Theorem \ref{EChiComplete}] \label{MainEchiComplete} If $\chi \in \mathcal W^+_p, \ p \geq 1$ then $(\mathcal E_{\tilde \chi}(X, \o),d_\chi)$ is a geodesic metric space, which is the metric completion of $(\mathcal H,d_\chi)$. Additionally, the curve defined in \eqref{EchiGeodDef} is a geodesic connecting $u_0,u_1$.
\end{main_theorem}

Although not specifically pursued there, from the ideas developed in \cite[Section 4]{g}, the toric version of this theorem easily follows. Using the Riemannian metric of Calabi \cite{cal, clm}, the completion of the space of K\"ahler metrics was first studied in \cite{cr}, and in Section 7.2 of this paper the authors proposed to study the completion with respect to other Riemannian metrics. The above theorem fits into this framework and in a future publication we will compare the Calabi and Mabuchi geometries.

The proof of Theorem \ref{MainEchiComplete} is given by adapting to our general setting the arguments of \cite[Theorem 1]{d}, which deals with the particular case of the $L^2$ metric. As a result of the condition $\chi \in \mathcal W^+_p$, few things need to be altered, and we will indicate only the necessary changes. In \cite{d} it is additionally proved that $(\mathcal E^2(X,\omega),d_2)$ is a CAT(0) space, hence geodesics connecting different points in this space are unique. As we shall see, none of these properties hold for general $\chi$ (see discussion after Theorem \ref{EChiComplete}).

\subsection{Characterization of the Path Length Metric Topology}

One would like to characterize the topology of $(\mathcal E_{\tilde \chi}(X,\omega),d_\chi)$ in very concrete terms. With this goal in mind, given $u_0,u_1 \in \mathcal E_{ \tilde \chi}(X,\o)$ we introduce the following energy functional:
$$I_\chi(u_0,u_1) = \|u_1 - u_0\|_{\chi,u_0} + \|u_1 - u_0\|_{\chi,u_1}.$$
This functional was first proposed in \cite{g} in the particular case of the $L^2$ metric. We say that $\{ u_j\}_{j \in \Bbb N} \subset \mathcal E_{\tilde\chi}(X,\o)$ converges to $u \in \mathcal E_{\tilde\chi}(X,\o)$ in $\chi$--energy if $\lim_j I_\chi(u_j,u)=0$. When $\chi(l)=\chi_1(l)=|l|$, this is equivalent to the notion of strong convergence introduced in \cite[Section 2.1]{bbegz} (see Proposition \ref{strongchiconveqv}). The next theorem says that convergence in $\chi$--energy and $d_\chi$--convergence are the same:
\begin{main_theorem}[Theorem \ref{Energy_Metric_Eqv_thm}] \label{Energy_Metric_Eqv}
Given $\chi \in \mathcal W^+_p, \ p \geq 1$, there exists $C(p) >1$ such that
\begin{equation}\label{mabenergeqv}\frac{1}{C}I_\chi(u_0-u_1) \leq d_{\chi}(u_0,u_1) \leq C I_\chi(u_0 - u_1), \ u_0,u_1 \in \mathcal E_{\tilde \chi}(X,\o).
\end{equation}
\end{main_theorem}
We note the following immediate consequence:
\begin{main_corollary} \label{supcorollary}Given $\chi \in \mathcal W^+_p$ and $u \in \mathcal E_{\tilde \chi}(X,\o)$ there exists $C(p) >1$ such that
$$\sup_X u \leq C d_\chi(0,u) +C.$$
\end{main_corollary}

\begin{proof} As all Orlicz norms dominate the $L^1$ norm \cite{rr}, the $d_1$ metric is dominated by all $d_\chi$ metrics and also $\mathcal E_{\tilde \chi}(X,\o) \subset \mathcal E^1(X,\o)$. Hence, it is enough to prove the result for $d_1$. It is known that for all $u \in \textup{PSH}(X,\o)$ there exists $C' >0$ such that $\sup_X u - C' \leq \int_X u \o^n.$
On the other hand we have $\int_X u \o^n \leq I_1(0,u) \leq C d_1(0,u)$ for $u \in \mathcal E^1(X,\o)$ by the above theorem.
\end{proof}

From the last theorem and the techniques developed in \cite{bbgz} and  \cite{bbegz} it follows that $d_\chi$--convergence implies more classical notions of convergence:
\begin{main_theorem}[Corollary \ref{MabConvCor}, Theorem \ref{FlatImmersion}] \label{MabConvergence} Suppose $\chi \in \mathcal W^+_p$ and $\ u_k,u \in \mathcal E_{\tilde \chi}(X,\o)$ with $d_\chi(u_k,u) \to 0$. Then the following holds:
\begin{itemize}
\item[(i)] $u_k \to u$ in capacity in the sense of \cite{k}. In particular, $\o_{u_k}^n \to \o_{u}^n$ weakly.
\item[(ii)] For any $v \in \mathcal E_{\tilde \chi}(X,\o)$ we have $\|u_k - u\|_{\chi,v} \to 0.$
\end{itemize}
\end{main_theorem}

Fixing $v \in \mathcal E_{\tilde \chi}(X,\o)$, by \cite[Proposition 3.6]{gz} it follows that $(\mathcal E_{\tilde \chi}(X,\o),d_\chi) \subset L^\chi(\o_v^n)$. The content of Theorem \ref{MabConvergence}(ii) is that this inclusion is continuous as a map between the following metric spaces:
$$(\mathcal E_{\tilde \chi}(X,\o),d_\chi) \to (L^\chi(\o_v^n),\| \cdot \|_{\chi,v}).$$

In the particular case of the $L^2$ metric, the last three results answer questions raised by R. Berman \cite{bpers} and V. Guedj \cite[Section 4.3]{g}.

\subsection{The $d_\chi$--convergence of the K\"ahler-Ricci trajectories}

Lastly, we give some immediate applications of our results to the stability of the K\"ahler--Ricci flow and K\"ahler--Einstein metrics. Let us assume that $(X,J,\o)$ is a Fano manifold. We introduce the following "totally geodesic" hypersurfaces of $\mathcal H$ and $\mathcal E^1(X,\o)$ (see \cite[Proposition 6.2]{bbgz}):
$$\mathcal H_{AM} = \mathcal H \cap \{ AM(\cdot)=0\},$$
$$\mathcal E^1_{AM}(X,\o) = \mathcal E^1(X,\o) \cap \{ AM(\cdot)=0\},$$
where $AM : \mathcal E^1(X,\o) \to \Bbb R$ is the Aubin-Mabuchi energy functional defined by the formula:
$$AM(v)=\frac{1}{(n+1)\textup{Vol}(X)}\sum_{j=0}^n\int_{X} v\o^j\wedge (\o + i\partial\bar\partial v)^{n-j}.$$
As we will see (Lemma \ref{AMcont}), $AM(\cdot)$ restricted to $\mathcal E_{\chi}(X,\o)$ is $d_\chi$--continuous for any $\chi \in \mathcal W^+_p$. Additionally, this functional is also linear along geodesics.

A smooth metric $\omega_{u_{KE}}$ is K\"ahler-Einstein if $\omega_{u_{KE}}=\textup{Ric }\omega_{u_{KE}}$. One can study such metric by looking at the long time asymptotics of Hamilton's K\"ahler--Ricci flow:
\begin{equation}\label{RicciflowEq}
\begin{cases}
\frac{d \o_{r_t}}{dt} = - \textup{Ric }\o_{r_t} + \o_{r_t}, \\
r_0 = v. \end{cases}
\end{equation}
As proved by Cao \cite{ca}, for any $v \in \mathcal H_{AM}$, this PDE has a smooth solution $[0,1) \ni t \to r_t \in \mathcal H_{AM}$. It follows from a theorem of Perelman and work of Chen-Tian, Tian-Zhu and Phong-Song-Sturm-Weinkove,
that whenever a K\"ahler--Einstein metric cohomologous to $\o$ exists, then $\o_{r_t}$ converges exponentially fast to one such metric (see \cite{ct,tz,pssw}).

We remark that our choice of normalization is different from the alternatives used in the literature (see \cite[Chapter 6]{beg}). We choose to work with the normalization $AM(\cdot)=0$, as this seems to be the most natural one from the point of view of Mabuchi geometry. As we shall see in Section 6, from the point of view of long time asymptotics, this choice is equivalent to all other alternatives.

Fix $\chi \in \mathcal W^+_p$. Motivated by the above, we would like to study the stability of K\"ahler--Ricci trajectories from the point of view of the $d_\chi$--geometry. In analogy with \cite[Definition 6.2]{cr}, we say that a Fano manifold $(X,J,\o)$ is $\chi$--unstable if there exists a starting point $v \in \mathcal H_{AM}$ and a corresponding K\"ahler--Ricci trajectory $[0,\infty) \ni t \to r_t \in \mathcal H_{AM}$ that diverges with respect to $d_\chi$. Otherwise, we say that $(X,J, \o)$ is $\chi$--stable.

Digressing slightly, we recall the definition of Ding's $\mathcal F$--functional and the $\mathcal J$--functional, which are as follows: $\mathcal F,\mathcal J : \mathcal E^1_{AM}(X,\o) \to \Bbb R$,
\begin{equation}\label{DIngDef}
\mathcal F(u) = - AM(u) - \log \int_X e^{-u + h}\o^n =  - \log \int_X e^{-u+h}\o^n,
\end{equation}
\begin{equation}\label{JDef}
\mathcal J(u) = \frac{1}{\textup{Vol}(X)}\int_X u \o^n - AM(u)=\frac{1}{\textup{Vol}(X)}\int_X u \o^n.
\end{equation}
Here $h \in C^\infty(X)$ is the Ricci potential of $\o$, i.e. $\textup{Ric }\o =\o+i\partial\bar\partial h$. By Theorem \ref{MabConvergence} above, both of these functionals are continuous with respect to $d_\chi$. By a result of  Tian \cite[Theorem 1.6]{t}, if $(X,J)$ admits no non-trivial holomorphic vector--fields, then $(X,J,\o)$ admits a K\"ahler-Einstein metric if and only if $\mathcal F$ is $\mathcal J$--proper, meaning that $\mathcal J(u_k) \to \infty$ implies $\mathcal F(u_k) \to \infty$. One would like to understand this condition from a more geometric point of view. Connecting all of the above, we give the following result:
\begin{main_theorem}[Theorem \ref{MR-stab}, Theorem \ref{Dingstab}] \label{ApplKE}Suppose $(X,J,\o)$ is a Fano manifold with no non--trivial holomorphic vector fields. The following are equivalent:
\begin{enumerate}
\item[(i)] There exists a K\"ahler--Einstein metric cohomologous to $\o$.
\item[(ii)] $(X,J, \o)$ is $\chi$--stable for any $\chi \in \mathcal W^+_p$.
\item[(iii)] $\mathcal F$ is $d_1$--proper on $\mathcal E^1_{AM}(X,\o)$, i.e. sublevel sets of $\mathcal F$ are $d_1$--bounded.
\end{enumerate}
\end{main_theorem}

With the help of Theorem \ref{MabConvergence}, the equivalence between (i) and (ii) can be given along the same lines as the analogous result for the Calabi metric \cite[Theorem 6.3]{cr}. As we shall see, the assumption that $(X,J)$ has no non--trivial holomorphic vector fields is not needed for this equivalence. In addition to this, whenever a K\"ahler--Einstein metric exists, not only is $(X,J, \o)$ is $\chi$--stable, but each Ricci trajectory $t \to r_t$ converges to one such metric exponentially fast. This strengthens results of McFeron \cite{mc} and answers questions raised in \cite[Section 4.3.4]{r}. We refer to this last expository work for connections with other conjectures in K\"ahler geometry.

The $d_1$--metric always dominates the Ding functional, i.e. there exists $A,B >0$ such that $\mathcal F(u) \leq A d_1(0,u) + B$ (Remark \ref{JensenRemark}). Because of this and the fact the each $d_\chi$ dominates $d_1$, it is not possible to replace $d_1$--properness in (iii) with the more general $d_\chi$--properness, as $d_1$ and $d_\chi$ are not equivalent in general.

In \eqref{1-pythagorean_formula} we will give the following explicit formula for the $d_1$ metric:
$$d_1(u_0,u_1) = AM(u_0) + AM(u_1) - 2AM(P(u_0,u_1)),\ u_0,u_1 \in \mathcal E^1(X,\o),$$
where $P(u_0,u_1) = \sup \{u: u \in \textup{PSH}(X,\o), \ u \leq u_0,u_1 \} \in \mathcal E^1(X,\o)$. This will immediately imply that $\mathcal J$--properness is equivalent to $d_1$--properness, which in turn gives the equivalence between (i) and (iii) by the above mentioned result of Tian.

Lastly, we mention that existence of K\"ahler--Einstein metrics on Fano manifolds is also known to be equivalent to the $\mathcal J$--properness of the Mabuchi K-energy functional restricted to $\mathcal H$. With a better understanding of the extension of this functional to $\mathcal E^1_{AM}(X,\o)$ one could hope for analogous results in terms of this functional instead of the Ding energy. For results in this direction we refer to \cite{bb2, st1,st2}.

\vspace{-0.1in}
\paragraph{Future work.} Continuing to explore the $d_\chi$--geometry of the K\"ahler--Ricci flow, in a future correspondence we will construct $d_\chi$--geodesic rays that are weakly asymptotic to diverging K\"ahler--Ricci trajectories, giving an affirmative answer to questions raised in \cite[Conjecture 4.10]{r}. These geodesic rays will have special properties, in particular the Ding energy is always decreasing along them. In case of the $L^2$ metric, this will help settle questions raised in \cite{h2}, where in analogy with Donaldson's original program on constant scalar curvature metrics \cite{do}, the equivalence between existence of K\"ahler--Einstein metrics and properness of the Ding energy along $d_2$--geodesic rays was proposed. As pointed out to us by R. Berman, his paper \cite{brm1}  and the results of \cite{cds1,cds2,cds3,t2} already imply this last equivalence.
Continuing the trend of the present work,  our methods will be purely analytical and will allow us to avoid the recently established equivalence between K-stability and existence of K\"ahler--Einstein metrics on Fano manifolds.
\vspace{0.1in}

\noindent \emph{Acknowledgements.} I would like to thank L. Lempert for his support and for his interest in this work. I also profited from discussions with R. Berman, V. Guedj, W. He and  Y. Rubinstein.

\section{Preliminaries}

\subsection{Orlicz spaces}

We record here some facts about Orlicz spaces that will be used the most. For a thorough introduction we refer to \cite{rr}. Suppose $(\Omega,\Sigma, \mu)$ is a Borel measure space with $\mu(\Omega)=1$ and $(\chi,\chi^*)$ is a complementary pair of Young weights. This means that $\chi:\Bbb R \to \Bbb R^+ \cup \{ \infty \}$ is convex, even, lower semi-continuous  and satisfies the normalizing conditions $\chi(0)=0$, $1 \in \partial \chi(1)$. Such $\chi$ alone is called a normalized Young weight. The complement $\chi^*$ is just the Legendre transform of $\chi$:
\begin{equation}
\chi^*(h) = \sup_{l \in \Bbb R} (lh - \chi(l)).
\end{equation}
One can see that $\chi^*$ is also a normalized Young weight and $(\chi,\chi^*)$ satisfies the Young identity and inequality:
\begin{equation}\label{YoungIdIneq}
\chi(\chi'(a)) + \chi^*(a)=a\chi'(a), \ \chi(a) + \chi^*(b) \geq ab, \ a,b \in \Bbb R.
\end{equation}
The most typical example to keep in mind is the pair $\chi_p(l)=|l|^p/p$ and $\chi(l)^*_p=|l|^q/q$, where $p,q > 1$ and $1/p + 1/q = 1$.
Let $L^\chi(\mu)$ be the following space of measurable functions:
$$L^\chi(\mu)=\Big\{ f:\Omega \to \Bbb R \cup \{ \infty,-\infty\}: \ \exists r >0 \textup{ s.t. } \int_\Omega \chi(rf) d\mu < \infty\Big\}.$$
One can introduce the following (gauge) norm on $L^\chi(\mu)$:
\begin{equation}\label{OrliczNormDef}
\|f\|_{\chi,\mu}=\inf \Big\{ r > 0 : \ \int_\Omega \chi\Big(\frac{f}{r}\Big) d\mu \leq \chi(1) \Big\}.
\end{equation}
The reference \cite{rr} uses the notation $N_\chi^\mu(f)$ instead, however in our context this notation seems to be more cumbersome when carrying out computations. The set $\{ \int_\Omega \chi(f) d\mu \leq \chi(1)\}$ is convex and symmetric in $L^\chi(\mu)$, hence $\|\cdot\|_{\chi,\mu}$ is nothing but the Minkowski seminorm of this set. One can verify that $\|f\|_{\chi,\mu}=0$ implies $f=0$ a.e. with respect to $\mu$, hence  $\|\cdot\|_{\chi,\mu}$ is a norm that makes $L^\chi(\mu)$ a complete metric space (see \cite[Theorem 3.3.10]{rr}). The reason we work with a complementary pair of Young weights is the H\"older inequality which follows from \eqref{YoungIdIneq}:
\begin{equation}\label{HolderIneq}
\int_X fg d\mu \leq \|f\|_{\chi,\mu} \|g\|_{\chi^*,\mu}, \ f \in L^\chi(\mu), \ g \in L^{\chi^*}(\mu).
\end{equation}
Coming back to examples, $L^{\chi_p}(\mu)=L^p(\mu)$ and $\|\cdot\|_{\chi^p,\mu}$ is just the usual $L^p$ norm.

Motivated by the study in \cite[Section 3]{gz}, we will be interested in normalized Young weights that are finite and satisfy the following growth estimate:
\begin{equation}\label{GrowthEst}
l\chi'(l) \leq p\chi(l), \ l >0,
\end{equation}
for some $p \geq 1$. For such weights we write $\chi \in \mathcal W^+_p$ (see \cite{gz}). We note that in the theory of Orlicz spaces the notation $\chi \in \Delta_2 = \bigcup_{p\geq 1} \mathcal W^+_p$ is used (see \cite{rr}). The main point is the following estimate :
\begin{proposition} For $\chi \in \mathcal W^+_p, \ p \geq 1$ and $0 < \varepsilon < 1$ we have
\begin{equation}\label{GrowthControl}
\varepsilon^p \chi(l) \leq \chi(\varepsilon l) \leq \varepsilon \chi(l), \ l > 0.
\end{equation}
\end{proposition}
\begin{proof}
The second estimate follows from convexity of $\chi$. For the first estimate we notice that for any $\delta >0$ we still have $\chi_\delta(l) = \chi(l) + \delta|l| \in \mathcal W^+_p$. As $\chi_{\delta}(h) >0$ for $h >0$, we can integrate $\chi'_\delta(h)/\chi_\delta(h) \leq p/h$ from $\varepsilon l$ to $l$ to obtain:
$$\varepsilon^p \chi_\delta(l) \leq \chi_\delta(\varepsilon l).$$
Letting $\delta \to 0$ the desired estimate follows.
\end{proof}

Estimate \eqref{GrowthControl} immediately implies that for $f \in L^\chi(\mu)$ we have
\begin{equation}\label{OrliczNormId}
\|f\|_{\chi,\mu}=\alpha>0 \textup{ if and only if } \int_\Omega \chi\Big( \frac{f}{\alpha}\Big)\mu=\chi(1).
\end{equation}

To simplify future notation, we introduce the bijective increasing functions $M_p(l) = \max\{l,l^p\}$, $m_p(l) = \min\{l,l^p\}$ for $p >0, l \geq 0$. Observe that $M_p \circ m_{1/p} = m_p \circ M_{1/p}=Id$. As a consequence of \eqref{GrowthControl}, we can establish the following very useful estimates:
\begin{proposition} \label{NormIntegralEst} If $\chi \in \mathcal W_p^+$  and $f \in L^\chi(\mu)$ then
\begin{equation}
 m_p(\|f\|_{\chi,\mu}) \leq \frac{\int_\Omega \chi(f)d\mu}{\chi(1)} \leq M_p(\|f\|_{\chi,\mu}).
\end{equation}
\begin{equation}
 m_{1/p}\Big(\frac{\int_\Omega \chi(f)d\mu}{\chi(1)}\Big) \leq \|f\|_{\chi,\mu} \leq M_{1/p}\Big(\frac{\int_\Omega \chi(f)d\mu}{\chi(1)}\Big).
\end{equation}
Hence, given a sequence $\{ f_j\}_{j \in \Bbb N}$, we have $\|f_j\|_{\chi,\mu} \to 0$ if and only if $\int_\Omega \chi(f_j)d\mu \to 0$ and $\|f_j\|_{\chi,\mu} \to N$ for $N > 0$ if and only if $\int_\Omega \chi(f_j/N)d\mu \to \chi(1)$.
\end{proposition}

In the future we will approximate Orlicz norms with weight in $\mathcal W^+_p$ with Orlicz norms having smooth weights. The following two approximation results, which are by no means optimal, will be very useful for us:

\begin{proposition} \label{approx_lemma}Suppose $\chi \in \mathcal W^+_p$ and $\{\chi_k \}_{k \in \Bbb N}$ is a sequence of normalized Young weights that converges uniformly on compacts to $\chi.$ Let $f$ be a bounded $\mu$--measurable function on $X$. Then $f \in L^\chi(\mu),L^{\chi_k}(\mu), \ k \in \Bbb N$ and we have that
$$\lim_{k \to +\infty}\| f\|_{\chi_k,\mu} = \| f\|_{\chi,\mu}.$$
\end{proposition}
\begin{proof} Suppose $N=\| f\|_{\chi,\mu}$. If $N=0$, then $f=0$ a.e. with respect to $\mu$ implying that $\| f\|_{\chi_k,\mu} = \| f\|_{\chi,\mu}=0$. So we assume that $N >0$. As $\chi \in \mathcal W^+_p$, by \eqref{OrliczNormId}, for any $\varepsilon >0$ there exists $\delta >0$ such that
$$\int_X \chi\Big(\frac{f}{(1+\varepsilon)N}\Big) d\mu<\chi(1)-2\delta< \chi(1) <\chi(1)+2\delta<\int_X \chi\Big(\frac{f}{(1-\varepsilon)N}\Big) d\mu$$
As $\chi_k$ tends uniformly on compacts to $\chi$, $f$ is bounded and $\mu(X)=1$, it follows from the dominated convergence theorem that for $k$ greater then some $k_0$ we have
$$\int_X \chi_k\Big(\frac{f}{(1+\varepsilon)N}\Big) d\mu<\chi_k(1)-\delta< \chi_k(1) <\chi_k(1)+\delta<\int_X \chi_k\Big(\frac{f}{(1-\varepsilon)N}\Big) d\mu$$
This implies that $(1-\varepsilon)N \leq  \| f\|_{\chi_k,\mu} \leq (1+\varepsilon)N$ for $k \geq k_0$, from which the conclusion follows.
\end{proof}
\begin{proposition} \label{approx_lemma2}Given $\chi \in \mathcal W^+_p$, there exists $\chi_k \in \mathcal W^+_{p_k} \cap C^\infty(\Bbb R),$ $k \in \Bbb N$ with $\{p_k\}_k$ possibly unbounded such that $\chi_k \to \chi$ uniformly on compacts.
\end{proposition}
\begin{proof} There are many ways to construct the sequence $\chi_k$. First we smoothen and normalize $\chi$, introducing the sequence $\tilde \chi_k$ in the process:
$$\tilde \chi_k(l) = (\delta_{k} \star \chi)(h_k l) - (\delta_{k} \star \chi)(0), \ l \in \Bbb R, k \in \Bbb N^*.$$
Here $\delta$ is a typical choice of bump function with support in $(-1,1)$, $\delta_k(\cdot) =k \delta(k(\cdot))$, and $h_k >0$ is chosen in such a way that $\tilde \chi_k$ becomes normalized ($\tilde \chi_k(0)=0$ and $\tilde \chi_k'(1)=1$). As $\chi$ is normalized it follows that $1-1/k \leq h_k \leq 1 +1/k$, hence the $\tilde \chi_k$ are normalized Young weights that converge to $\chi$ uniformly on compacts.

As $\chi$ is strictly increasing, so is each $\tilde \chi_k$, but
our construction does not seem to guarantee that $\tilde \chi_k \in \mathcal W^+_{p_k}$ for some $p_k \geq 1$. This can be fixed by changing smoothly the values of $\tilde \chi_k$ on the sets $|x| \leq 1/k$ and $|x| > k$ in such a manner that that the possibly altered weights $\chi_k$ satisfy the estimate
$$l \chi'_k(l) \leq p_k \chi_k(l), \ l >0,$$
for some $p_k  \geq 1$. The sequence we obtained satisfies the required properties.
\end{proof}

\subsection{Weak Geodesics in the Metric Space $(\mathcal H,d_2)$}

We summarize some of the properties of the metric space $(\mathcal H,d_2)$ that we will need later, for a thorough introduction we refer to \cite{bl1}.  If in \eqref{FinslerDef} we choose $\chi(l)=\chi_2(l)=l^2/2$, we obtain the Mabuchi Riemannian metric on $\mathcal H$. Given a smooth curve $(0,1) \ni t \to v_t \in \mathcal H$ and a vector field $(0,1) \ni t \to f_t \in C^\infty(X)$ along this curve, the covariant derivative $\nabla_{\dot u_t} f_t$ is as follows:
\begin{equation}\label{CovDerivative}
\nabla_{\dot u_t} f_t = \dot f_t - \frac{1}{2}\langle \nabla^{\o_{u_t}} \dot u_t , \nabla^{\o_{u_t}} f_t \rangle_{\o_{u_t}}
\end{equation}

Suppose $S = \{ 0 < \textup{Re }s < 1\}\subset \Bbb C$. Following \cite{s}, one can compute that a smooth curve $[0,1] \ni t \to u_t \in \mathcal H$ connecting $u_0,u_1 \in \mathcal H$ is a geodesic if it is the (unique) smooth solution of the  following Dirichlet problem on $S \times X$:
\begin{alignat}{2}\label{BVPGeod}
&(\pi^* \o + i \partial \overline{\partial}u)^{n+1}=0 \nonumber\\
&u(t+ir,x) =u(t,x) \ \forall x \in X, t \in (0,1), r \in \Bbb R \\
&\lim_{t \to 0,1}u_t=u_{0,1} \textup{ uniformly in }X,\nonumber
\end{alignat}
where $u(s,x)=u_{\textup{Re }s}(x)$ is the complexification of $t \to u_t$. Unfortunately the above problem does not usually have smooth solutions (see \cite{lv,da3}), but a unique solution in the sense of Bedford-Taylor does exist and with bounded Laplacian (see \cite{c}) and this regularity is essentially optimal (see \cite{dl}). This curve $[0,1] \ni t \to u_t \in \mathcal H_\Delta := \textup{PSH}(X,\omega)\cap \{ \Delta v \in L^\infty\}$ is called the weak geodesic connecting $u_0,u_1$, and it can be approximated by smooth $\varepsilon$--geodesics which are solutions of the following elliptic Dirichlet problem on $[0,1]\times X$:
\begin{alignat}{2}\label{BVPEpsGeod}
&(\ddot{u^\varepsilon_t} - \frac{1}{2} \langle \nabla\dot u^\varepsilon_t, \nabla\dot u^\varepsilon_t\rangle)(\o + i\partial\bar\partial u^\varepsilon_t)^n=\varepsilon \o^n, \ t \in [0,1].\\
&\lim_{t \to 0,1}u_t=u_{0,1} \textup{ uniformly in }X,\nonumber
\end{alignat}
Chen proved that the path length metric can be computed using $\varepsilon$--geodesics:
\begin{equation}\label{DistMabuchiApr}
d_2(u_0,u_1)=\lim_{\varepsilon \to 0} l_\varepsilon(u^\varepsilon),
\end{equation}
and there exists $C>0$ independent of $\varepsilon$ such that
\begin{equation}\label{EpsGeodLaplEst}
\| \Delta u^\varepsilon\|_{L^\infty([0,1]\times X)} \leq C.
\end{equation}
From this it follows that the $\varepsilon$--geodesics $u^\varepsilon$ converge to the weak geodesic $u$ in $C^{1,\alpha}(\overline{S}\times X)$, $0 < \alpha < 1$ as $\varepsilon \to 0$.
Using this, one can relate $d_2(u_0,u_1)$ directly to the weak geodesic $t \to u_t$:
\begin{equation}\label{distgeod}
d_2(u_0,u_1) = \| \dot u_t\|_{2,u_t}=\sqrt{\frac{1}{\textup{Vol}(X)}\int_X\dot {u}_{t}^2\o_{u_{t}}^n}, \ t \in[0,1].
\end{equation}
This formula is geometrically justified, as in finite dimensional Riemannian geometry geodesics have constant speed. Although weak geodesics connecting points of $\mathcal H$ leave this space, they are bona fide geodesics in the metric completion of $(\mathcal H, d_2)$ (see \cite{d}).

Given $u_0,u_1 \in \mathcal H$ one can introduce the rooftop-envelope $P(u_0,u_1)$:
$$P(u_0,u_1)=\sup\{ v \in \text{PSH}(X,\o) \ | \ v \leq \min \{ u_0,u_1 \}\}.$$
As observed in \cite{dr}, one has  $P(u_0,u_1)\in \mathcal H_\Delta$, and we have the following partition formula for the volume form:
\begin{equation}\label{MA_formula}
\o_{P(u_0,u_1)}^n= \mathbbm{1}_{\{u_0 = P(u_0,u_1)\}}\o_{u_0}^n + \mathbbm{1}_{\{u_1 = P(u_0,u_1)\} \setminus \{u_0 = P(u_0,u_1)\}}\o_{u_1}^n.
\end{equation}
Moreover, these envelopes are intimately connected to the weak geodesic $[0,1] \ni t \to u_t \in \mathcal H_\Delta$, joining $u_0,u_1$ in the following manner:
\begin{equation}\label{WeakGeodEnvID}
\inf_{t \in [0,1]} (u_t - \tau t) = P(u_0,u_1 - \tau),
\end{equation}
\begin{equation}\label{SublevelSetId}
\{ \dot u_0 \geq \tau \} = \{ P(u_0,u_1 - \tau)=u_0\},
\end{equation}
for $\tau \in \Bbb R$. For the proof of these facts we refer to \cite[Section 2.2 and Lemma 6.5]{d} and \cite{dr}.

Developing \eqref{distgeod} further, by a result of Berndtsson \cite{br2}, the pushforward measures $\dot {u_t}_{*} \o_{u_t}$ along the weak geodesic segment are the same for any $t \in [0,1]$. As a consequence of this we obtain that weak geodesics have constant speed with respect to all $\chi$--Finsler metrics:
\begin{remark} \label{chilengthgeodremark} Suppose $[0,1] \ni t \to u_t \in \mathcal H_\Delta$ is the weak geodesic connecting $u_0,u_1 \in \mathcal H$ and $\chi\in\mathcal W^+_p$. Then we have
$$\|\dot u_0\|_{\chi,u_0}=\|\dot u_t\|_{\chi,u_t}, \ t \in [0,1].$$
\end{remark}
\begin{proof} Suppose $N = \|\dot u_0\|_{\chi,u_0}$. If $N=0$ then $\dot u_0 =0$, hence $d_2(u_0,u_1)=0$ by \eqref{distgeod}. This in turn implies $u_0=u_1$, finishing the proof in this case. If $N \neq 0$, then by the above mentioned result of Berndtsson it follows that
$$\chi(1) = \frac{1}{\textup{Vol}(X)}\int_X \chi\Big( \frac{\dot u_0}{N}\Big)\o_{u_0}^n = \frac{1}{\textup{Vol}(X)}\int_X \chi\Big( \frac{\dot u_t}{N}\Big)\o_{u_t}^n, \ t \in [0,1].$$
Hence, $N=\|u_0\|_{\chi,u_0}= \|u_t\|_{\chi,u_t}$.
\end{proof}
This remark suggests that weak geodesics have special role, not just in the $L^2$ geometry, but in the $\chi$--Finsler geometry of $\mathcal H$ as well. This is indeed the case, as we will see shortly.

\section{The Metric Spaces $(\mathcal H, d_\chi)$}

In this section we give the proof of Theorem \ref{XXChenThm}. Our method will follow Chen's original approach (\cite{c}, see \cite[Section 6]{bl1} for a recent survey) along with a careful analysis involving Orlicz norms. In hopes of easing the technical nature of future calculations, for Sections 3--5  we fix the volume normalizing condition
$$\textup{Vol}(X)=\int_X \o^n = 1.$$

Given a differentiable normalized Young weight $\chi$, the differentiability of the associated gauge norm is well understood (see \cite[Chapter VII]{rr}). Adapted to our setting we prove the following result which is essentially contained in \cite[Theorem VII.2.3]{rr}.

\begin{proposition}\label{OrliczNormDiff} Suppose $\chi \in \mathcal W^+_p \cap C^\infty(\Bbb R)$. Given a smooth curve $(0,1) \ni t \to u_t \in \mathcal H$, i.e. $u(t,x):=u_t(x) \in C^\infty((0,1)\times X)$, and a non-vanishing vector field $(0,1) \ni t \to f_t \in C^\infty(X)$ along this curve, the following formula holds:
\begin{equation}\label{OrliczNormDiffEq}
\frac{d}{dt}\|f_t\|_{\chi,u_t} = \frac{\int_X \chi' \Big(\frac{f_t}{\|f_t\|_{\chi,u_t}}\Big)\nabla_{\dot u_t}\dot f_t \omega_{u_t}^n}{\int_X \chi'\Big(\frac{f_t}{\|f_t\|_{\chi,u_t}}\Big) \frac{f_t}{\|f_t\|_{\chi,u_t}}  \omega_{u_t}^n},
\end{equation}
where $\nabla$ is the covariant derivative from \eqref{CovDerivative}.
\end{proposition}
\begin{proof}
We introduce the smooth function $F: \Bbb R^+ \times (0,1) \to \Bbb R$ given by
$$F(r,t) = \int_X \chi\Big(\frac{f_t}{r}\Big)\o_{u_t}^n.$$
By our assumption on $\chi$ we have $\chi'(l) >0, \ l >0$. As $t \to f_t$ is non-vanishing, it follows that
$$\frac{d}{dr} F(r,t)=- \frac{1}{r^2}\int_X {f_t} \chi'\Big(\frac{f_t}{r}\Big) \omega_{u_t}^n < 0$$
for all $r <0,\ t \in (0,1)$. Using
$F(\|f_t\|_{\chi,u_t},t) = \chi(1)$, an application of the implicit function theorem yields that the map $t \to \|f_{t}\|_{\chi,u_t}$ is differentiable and the following formula holds:
$$\frac{d}{dt}\|f_t\|_{\chi,u_t} =
\frac{\int_X \Big[\dot f_t \chi'\Big(\frac{f_t}{\|f_t\|_{\chi,u_t}}\Big) + \|f_t\|_{\chi,u_t} \chi\Big(\frac{f_t}{\|f_t\|_{\chi,u_t}}\Big) \Delta_{\omega_{u_t}} \dot u_t \Big] \omega_{u_t}^n}{\int_X \frac{f_t}{\|f_t\|_{\chi,u_t}} \chi'\Big(\frac{f_t}{\|f_t\|_{\chi,u_t}}\Big) \omega_{u_t}^n}.$$
An integration by parts yields \eqref{OrliczNormDiffEq}.
\end{proof}

\begin{corollary} \label{EpsGeodDiffCor} Suppose $\chi \in \mathcal W^+_p \cap C^\infty(\Bbb R)$ and $v_0,v_1 \in \mathcal H, \ v_0 \neq v_1$. Then there exists $\varepsilon_0(v_0,v_1) >0$ such that for any $\varepsilon < \varepsilon_0$ and $u_0,u_1 \in \mathcal H$ satisfying $\| u_0 - v_0\|_{C^2(X)}, \| u_1 - v_1\|_{C^2(X)} \leq \varepsilon_0$, the $\varepsilon$--geodesic $[0,1] \ni t \to u_t \in \mathcal H$ of \eqref{BVPEpsGeod}, connecting $u_0,u_1$ satisfies:
\begin{equation}\label{EpsGeodDiffEq}
\frac{d}{dt}\|\dot u_t\|_{\chi,u_t}=\varepsilon\frac{\int_X \chi'\Big(\frac{\dot u_t}{\|\dot u_t\|_{\chi,u_t}}\Big)\o^n} {\int_X \frac{\dot u_t}{\|\dot u_t\|_{\chi,u_t}} \chi'\Big(\frac{\dot u_t}{\|\dot u_t\|_{\chi,u_t}}\Big) \omega_{u_t}^n}, \ t \in [0,1].
\end{equation}
\end{corollary}
\begin{proof}From \cite[Lemma 13]{bl1} and the estimates of \cite[Theorem 12]{bl1} it follows that there exists $\varepsilon_0(v_0,v_1) >0$ such that for any $\varepsilon < \varepsilon_0$ and $u_0,u_1 \in \mathcal H$ satisfying $\| u_0 - v_0\|_{C^2(X)}, \| u_1 - v_1\|_{C^2(X)} \leq \varepsilon_0$ the $\varepsilon$--geodesic $(0,1) \ni t \to u_t \in \mathcal H$ connecting $u_0,u_1$ satisfies the non-vanishing condition $\dot u_t \not\equiv 0, \ t \in [0,1]$. Formula \eqref{EpsGeodDiffEq} is now just a consequence of the prior proposition and \eqref{BVPEpsGeod}.
\end{proof}

Continuing to focus on smooth weights $\chi$, we establish a concrete lower bound on the $\chi$--length of tangent vectors along the $\varepsilon$--geodesics defined in \eqref{BVPEpsGeod}. This is the analog of \cite[Lemma 13]{bl1} in our more general setting.

\begin{proposition} \label{EpsGeodTanEst} Suppose $\chi \in \mathcal W^+_p \cap C^\infty(\Bbb R)$ and $v_0,v_1 \in \mathcal H, \ v_0 \neq v_1$. Then there exists $\varepsilon_0(\chi,v_0,v_1) >0$ and $R(\chi,v_0,v_1) >0$ such that for any $\varepsilon < \varepsilon_0$ and $u_0,u_1 \in \mathcal H$ satisfying $\| u_0 - v_0\|_{C^2(X)}, \| u_1 - v_1\|_{C^2(X)} \leq \varepsilon_0$, the $\varepsilon$--geodesic $[0,1] \ni t \to u_t \in \mathcal H$ connecting $u_0,u_1$ satisfies:
\begin{equation}\label{dotintegralest}\int_X \chi( \dot u_t ) \o_{u_t}^n \geq \max \{\int_X \chi(\min(u_1 - u_0,0))\o_{u_0}^n,\int_X \chi(\min(u_0 - u_1,0))\o_{u_1}^n\} -\varepsilon R > 0,
\end{equation}
$t \in [0,1]$. Additionally, there also exists $R_0(\varepsilon_0,\chi, v_0,v_1),R_1(\varepsilon_0,\chi, v_0,v_1)>0$ such that
\begin{itemize}
 \item[(i)]$\| \dot u_t\|_{\chi,u_t} > R_0, \ t \in [0,1]$,
 \item[(ii)] $\Big|\frac{d}{dt}\|\dot u_t\|_{\chi,u_t}\Big| \leq \varepsilon R_1, \ t \in [0,1]$.
\end{itemize}
\end{proposition}
\begin{proof}
Suppose $u_0,u_1 \in \mathcal H$. As $t \to u_t(x)$ is convex for any $x \in X$, on the set $\{ u_0 \geq u_1 \}$ the estimate $\dot u_0 \leq u_1 - u_0\leq 0$ holds, hence
$$\int_X \chi( \dot u_0 ) \o_{u_0}^n \geq \int_X \chi(\min(u_1 - u_0,0))\o_{u_0}^n.$$
We can similarly deduce that
$$\int_X \chi( \dot u_1 ) \o_{u_1}^n \geq \int_X \chi(\min(u_1 - u_0,0))\o_{u_1}^n.$$
For $t \in [0,1]$, using the fact that $\nabla$ is a Riemannian connection and $t \to u_t$ is an $\varepsilon$--geodesic, we can write:
$$\label{etangentest}
\Big|\frac{d}{dt} \int_X \chi(\dot u_t) \o_{u_t}^n\Big|=\Big|\int_X \chi'(\dot u_t) \nabla_{\dot u_t} \dot u_t\o_{u_t}^n\Big|= \varepsilon \Big|\int_X \chi'(\dot u_t) \o^n\Big| \leq \varepsilon R(\chi, u_0,u_1),$$
where in the last estimate we have used that $\dot u_t$ is uniformly bounded in terms of  $\| u_0\|_{C^2}, \| u_1\|_{C^2}$. Putting together the last three estimates it is clear that for some small $\varepsilon_0(\chi,v_0,v_1)$ the estimate of \eqref{dotintegralest} holds. Using Proposition \ref{NormIntegralEst}, (i) follows from \eqref{dotintegralest}.

To establish (ii) we shirnk $\varepsilon_0$ further to satisfy the requirements of Corollary \eqref{EpsGeodDiffCor}. Using the Young identity \eqref{YoungIdIneq} we can write:
\begin{flalign}\label{epstangentest}
\Big|\frac{d}{dt}\|\dot u_t\|_{\chi,u_t}\Big|&=\varepsilon\frac{\Big|\int_X \chi'\Big(\frac{\dot u_t}{\|\dot u_t\|_{\chi,u_t}}\Big)\o^n\Big|} {\int_X \frac{\dot u_t}{\|\dot u_t\|_{\chi,u_t}} \chi'\Big(\frac{\dot u_t}{\|\dot u_t\|_{\chi,u_t}}\Big) \omega_{u_t}^n}=\varepsilon\frac{\Big|\int_X \chi'\Big(\frac{\dot u_t}{\|\dot u_t\|_{\chi,u_t}}\Big)\o^n\Big|} {\chi(1)+\int_X \chi^*\Big(\chi'\Big(\frac{\dot u_t}{\|\dot u_t\|_{\chi,u_t}}\Big)\Big) \omega_{u_t}^n}\leq \nonumber \\
&\leq \frac{\varepsilon}{\chi(1)}\int_X \chi'\Big(\frac{\dot u_t}{\|\dot u_t\|_{\chi,u_t}}\Big)\o^n.
\end{flalign}
Using (i) and the fact that $\dot u_t$ is uniformly bounded in terms of  $\| u_0\|_{C^2}, \| u_1\|_{C^2}$ the estimate of (ii) follows.
\end{proof}

With the help of  \eqref{EpsGeodDiffEq} we can establish an estimate for $\varepsilon$--geodesics which is the analog of \cite[Theorem 14]{bl1} in our more general setting.
\begin{proposition} Suppose $\chi \in \mathcal W^+_p \cap C^\infty(\Bbb R)$, $[0,1] \ni s \to \psi_s \in \mathcal H$ is a smooth curve, $\phi \in \mathcal H \setminus \psi([0,1])$ and  $\varepsilon >0$.  We denote by $u \in C^\infty([0,1]\times [0,1] \times X)$ the smooth function for which $[0,1] \ni t \to u(t,s,\cdot) \in \mathcal H$ is the $\varepsilon$--geodesic connecting $\phi$ and $\psi_s, \ s \in [0,1]$. There exists $\varepsilon_0(\psi,\phi) >0$ such that for any $\varepsilon \leq \varepsilon_0$ the following holds:
$$l_\chi(u(\cdot,0)) \leq l_\chi(\psi) + l_\chi(u(\cdot,1)) + \varepsilon R,$$
for some $R(\phi,\psi,\chi,\varepsilon_0) >0$ independent of $\varepsilon >0$.
\end{proposition}
\begin{proof}Fix $s \in [0,1]$. To avoid cumbersome notation, derivatives in the $t$--direction will be donoted by dots, derivatives in the $s$--direction will be denoted by $d/ds$ and sometimes we omit dependence on $(t,s)$.  By Corollary \ref{EpsGeodDiffCor} for $\varepsilon_0(\psi,\phi)>0$ small enough we can write:
\begin{flalign*}
\frac{d}{ds}l_\chi(u(\cdot,s))&= \int_0^1 \frac{d}{ds} \|\dot u(t,s))\|_{\chi,u(t,s)}dt = \int_0^1 \frac{\int_X \chi' \Big(\frac{\dot u}{\|\dot u\|_{\chi,u}}\Big)\nabla_{\frac{du}{ds}}\dot u \omega_{u}^n}{\int_X \chi'\Big(\frac{\dot u}{\|\dot u\|_{\chi,u}}\Big) \frac{\dot u}{\|\dot u\|_{\chi,u}}  \omega_{u}^n}dt
\end{flalign*}
Using the Young identity \eqref{YoungIdIneq} and the fact that $\nabla$ is a Riemannian connection, we can continue:
\begin{flalign}\label{anothercalc}
&=\int_0^1 \frac{\int_X \chi' \Big(\frac{\dot u}{\|\dot u\|_{\chi,u}}\Big)\nabla_{\frac{du}{ds}}\dot u \omega_{u}^n}{\chi(1) + \int_X \chi^*\Big(\chi'\Big(\frac{\dot u}{\|\dot u\|_{\chi,u}}\Big)\Big)\omega_{u}^n}dt \nonumber\\
&=\int_0^1 \frac{\int_X \chi' \Big(\frac{\dot u}{\|\dot u\|_{\chi,u}}\Big)\nabla_{\dot u}\frac{du}{ds} \omega_{u}^n}{\chi(1) + \int_X \chi^*\Big(\chi'\Big(\frac{\dot u}{\|\dot u\|_{\chi,u}}\Big)\Big)\omega_{u}^n}dt\nonumber\\
&=\int_0^1 \frac{\frac{d}{dt}\int_X \chi' \Big(\frac{\dot u}{\|\dot u\|_{\chi,u}}\Big)\frac{du}{ds}\omega_{u}^n -\int_X \frac{du}{ds} \nabla_{\dot u}\Big(\chi' \Big(\frac{\dot u}{\|\dot u\|_{\chi,u}}\Big)\Big) \omega_{u}^n}{\chi(1) + \int_X \chi^*\Big(\chi'\Big(\frac{\dot u}{\|\dot u\|_{\chi,u}}\Big)\Big)\omega_{u}^n}dt.
\end{flalign}
We make the following side computation:
\begin{equation}\label{interimnabla}
\nabla_{\dot u}\Big(\chi' \Big(\frac{\dot u}{\|\dot u\|_{\chi,u}}\Big)\Big)\o_{u}^n=\chi'' \Big(\frac{\dot u}{\|\dot u\|_{\chi,u}}\Big)\Big( \frac{\nabla_{\dot u}\dot u}{\|\dot u\|_{\chi,u}} - \frac{1}{\|\dot u\|_{\chi,u}^2}\frac{d}{dt}\|\dot u\|_{\chi,u}\Big)\o_u^n
\end{equation}
After possibly further shrinking $\varepsilon_0(\phi,\psi) >0$, from Proposition \ref{EpsGeodTanEst}(i)(ii) and \eqref{BVPEpsGeod} it follows that $\|\dot u\|_{\chi,u}$ is uniformly bounded away from zero and both $\nabla_{\dot u} \dot u\o_u^n$ and $\frac{d}{dt} \|\dot u \|_{\chi,u}$ are of the form $\varepsilon R$, where $R$ is an uniformly bounded quantity for $\varepsilon < \varepsilon_0(\phi,\psi)$. Furthermore, it follows from Chen's arguments that $\dot u$ and $du/ds$ are uniformly bounded independently of $\varepsilon$ (see again \cite[Theorem 12]{bl1}). All of this implies that the quantity of \eqref{interimnabla} is also of the form $\varepsilon R$. Building on this, the second term in the numerator of \eqref{anothercalc} can be estimated and we can continue to write:
\begin{flalign*}
&= \int_0^1 \frac{\frac{d}{dt}\int_X \chi' \Big(\frac{\dot u}{\|\dot u\|_{\chi,u}}\Big)\frac{du}{ds} \o_{u}^n}{\chi(1) + \int_X \chi^*\Big(\chi'\Big(\frac{\dot u}{\|\dot u\|_{\chi,u}}\Big)\Big)\o_{u}^n}dt +\varepsilon R
\end{flalign*}
As $\chi^*$ is the Legendre transform of $\chi$, it follows that ${\chi^*}' (\chi'(l))=l, \ l \in \Bbb R$. Using this, our prior observations and the chain rule, we obtain that the expression
$$\frac{d}{dt}\left( \chi(1)+\int_X \chi^*\Big(\chi'\Big(\frac{\dot u}{\|\dot u\|_{\chi,u}}\Big)\Big)\o_{u}^n\right)=\int_X \frac{\dot u}{\|\dot u\|_{\chi,u}} \chi''\Big(\frac{\dot u}{\|\dot u\|_{\chi,u}}\Big)\nabla_{\dot u}\Big(\frac{\dot u}{\|\dot u\|_{\chi,u}}\Big)\o_{u}^n$$
is again of magnitude $\varepsilon R$, hence in our sequence of calculations we can write
\begin{flalign}\label{lastestimate}
&= \int_0^1 \frac{d}{dt} \frac{\int_X \chi' \Big(\frac{\dot u}{\|\dot u\|_{\chi,u}}\Big)\frac{du}{ds} \omega_{u}^n}{\chi(1) + \int_X \chi^*\Big(\chi'\Big(\frac{\dot u}{\|\dot u\|_{\chi,u}}\Big)\Big)\omega_{u}^n}dt +\varepsilon R \nonumber\\
&=\frac{\int_X \chi' \Big(\frac{\dot u(1,s)}{\|\dot u(1,s)\|_{\chi,\psi}}\Big)\frac{d \psi(s)}{ds} \omega_{\psi}^n}{\chi(1) + \int_X \chi^*\Big(\chi'\Big(\frac{\dot u(1,s)}{\|\dot u(1,s)\|_{\chi,\psi}}\Big)\Big)\omega_{\psi}^n} + \varepsilon R \nonumber\\
&\geq -\Big\|\frac{d\psi(s)}{ds}\Big\|_{\chi, \psi} + \varepsilon R,
\end{flalign}
where in the last line we have used the Young inequality \eqref{YoungIdIneq} in the following manner:
\begin{flalign*}
\frac{{\int_X \chi' \Big(\frac{\dot u(1,s)}{\|\dot u(1,s)\|_{\chi,\psi}}\Big)\frac{d \psi(s)}{ds} \omega_{\psi}^n}}{\| {d \psi}/{ds}\|_{\chi,\psi}}&\geq- \int_X \Big[ \chi \Big( \frac{{d \psi}/{ds}}{\| {d \psi}/{ds}\|_{\chi,\psi}}\Big) + \chi^*\Big(\chi' \Big(\frac{\dot u(1,s)}{\|\dot u(1,s)\|_{\chi,\psi}}\Big)\Big) \Big]\omega_{\psi}^n\\
&=-\Big[\chi(1) + \int_X \chi^*\Big(\chi' \Big(\frac{\dot u(1,s)}{\|\dot u(1,s)\|_{\chi,\psi}}\Big)\Big) \omega_{\psi}^n\Big].
\end{flalign*}
Integrating estimate \eqref{lastestimate} with respect to $s$ yields the desired inequality.
\end{proof}
With all the ingredients of the proof in place we can establish the main result of this section:
\begin{theorem} \label{ChiChenTheorem}Suppose $\chi\in \mathcal W^+_p$ and $[0,1] \ni t \to u_t \in \mathcal H_\Delta$ is  the weak geodesic of \eqref{BVPEpsGeod} joining $u_0,u_1 \in \mathcal H$. Then we have
\begin{equation}\label{ChiDistGeodFormula}
d_\chi(u_0,u_1) = l_\chi(u)= \|\dot u_t\|_{\chi,u_t}, \ t \in [0,1].
\end{equation}
Consequently, $(\mathcal H,d_\chi)$ is a metric space.
\end{theorem}

\begin{proof} As the smooth $\varepsilon$--geodesics $u^\varepsilon$ connecting $u_0,u_1$ converge to $u$ in $C^{1,\alpha}(\overline S \times X)$ we have
$$\lim_{\varepsilon \to 0} l_\chi(u^{\varepsilon}) = l_\chi(u),$$
hence $d_\chi(u_0,u_1) \leq l_\chi(u)$.

For the other inequality, we assume first that $\chi \in \mathcal W^+_p \cap C^\infty$. We have to prove that
\begin{equation}\label{reverseineq}
l_\chi(\phi) \geq l_\chi(u)
\end{equation}
for all smooth curves $[0,1] \ni t \to \phi_t \in \mathcal H$ connecting $u_0,u_1$. We can assume that $u_1 \not \in \phi[0,1)$ and let $h \in [0,1)$. Letting $\varepsilon \to 0$ in the previous result we obtain that
$$l_\chi(v) \leq l_\chi(\phi|_{[0,h]}) + l_\chi(w^h),$$
where $(0,1) \ni t \to v_t,w^h_t \in \mathcal H_\Delta$ are the weak geodesic segments joining $u_1,u_0$ and $u_1,\phi_h$ respectively. As $h \to 1$ we have $l_\chi(w^h)\to 0$ and we obtain \eqref{reverseineq}.
For general $\chi \in \mathcal W^+_p$ by Proposition \ref{approx_lemma2} there exists a sequence $\chi_k \in \mathcal W^+_{p_k} \cap C^\infty(\Bbb R)$ such that $\chi_k$ converges to $\chi$ uniformly on compacts. From what we just proved it follows that
$$\int_0^1 \| \dot \phi_t\|_{\chi_k,\phi_t}dt=l_{\chi_k}(\phi) \geq l_{\chi_k}(u)=\int_0^1 \| \dot u_t\|_{\chi_k,u_t}dt.$$
Using Proposition \ref{approx_lemma} and the dominated convergence theorem ($\dot \phi_t, \dot u_t$ are uniformly bounded), we can take the limit in this last estimate to conclude \eqref{reverseineq}.
Formula \eqref{ChiDistGeodFormula} follows now from Remark \ref{chilengthgeodremark}.

Finally, if $u_0\neq u_1$ then $\dot u_0 \not\equiv 0$, hence $d_\chi(u_0,u_1)=\| \dot u_0\|_{\chi,u_0}>0$. This implies that $(\mathcal H, d_\chi)$ is a metric space.
\end{proof}

\section{The Metric Spaces $(\mathcal E_{\tilde \chi}(X,\omega), d_\chi)$}

For the entire section we fix $\chi \in \mathcal W^+_p$. Our first result is the analog of \cite[Proposition 2.16]{g}in our more general setting:

\begin{lemma} \label{Mdist_est}Suppose $u_0,u_1 \in \mathcal H$ with $u_0 \leq u_1$. We have:
$$\max\{2^{-n-2} \|u_1-u_0\|_{\chi,u_0}, \|u_1-u_0\|_{\chi,u_1}\} \leq d_\chi(u_0,u_1) \leq \|u_1-u_0\|_{\chi,u_0}.$$
\end{lemma}

\begin{proof} Suppose $(0,1) \ni t \to u_t \in \mathcal H_\Delta$ is the weak geodesic segment joining $u_0$ and $u_1$. By (\ref{distgeod}) we have
$$d_\chi(u_0,u_1)=\|\dot u_1\|_{\chi,u_1}=\|\dot u_0\|_{\chi,u_0}.$$ Since $u_0 \leq u_1$, we have that $u_0 \leq u_t$. Since $(t,x) \to u_t(x)$ is convex in the $t-$variable, it results that $0 \leq \dot u_0 \leq u_1-u_u \leq \dot u_1$ and
\begin{equation}\label{NaiveDistEst}
\|u_1-u_0\|_{\chi,u_1} \leq d_\chi(u_0,u_1) \leq \|u_1-u_0\|_{\chi,u_0}
\end{equation}
follows.

We introduce $N := \|u_1 - u_0\|_{\chi,u_0}$. Using $ \o_{u_0}^n \leq 2^{n}\o_{(u_0+u_1)/2}^n$ and the convexity of $\chi$ we have
\begin{flalign*}
\chi(1)=&\int_X \chi\Big( \frac{u_1 - u_0}{N}\Big) \o_{u_0}^n\leq \int_X \chi\Big( \frac{u_1 - (u_0 + u_1)/2}{N/2}\Big) 2^{n}\o_{(u_0+u_1)/2}^n\\
\leq &\int_X \chi\Big( \frac{u_1 - (u_0 + u_1)/2}{N/2^{n+1}}\Big) \o_{(u_0+u_1)/2}^n.
\end{flalign*}
From the definition of the Orlicz norm it follows that
$$\frac{\|u_1 - u_0\|_{\chi,u_0}}{2^{n+1}} \leq \Big\|u_1 - \frac{u_0 + u_1}{2}\Big\|_{\chi,\frac{u_0 + u_1}{2}}=\Big\|u_0 - \frac{u_0 + u_1}{2}\Big\|_{\chi,\frac{u_0 + u_1}{2}}.$$
The first estimate of \eqref{NaiveDistEst} allows us to continue and obtain:
$$\frac{\|u_1 - u_0\|_{\chi,u_0}}{2^{n+1}} \leq d_\chi\Big(\frac{u_0 + u_1}{2},u_0\Big).$$
But we have $d_\chi((u_0 + u_1)/2,u_0) \leq d_\chi(u_0,u_1) + d_\chi((u_0 + u_1)/2,u_1)$, and $d_\chi((u_0 + u_1)/2,u_1) \leq d_\chi(u_0,u_1)$ as follows from the lemma below. This implies the desired estimate.
\end{proof}

\begin{lemma} \label{NaiveCompare}Suppose $u,v,w \in \mathcal H$ with $u \geq v \geq w$. Then we have $d_\chi(u,v) \leq d_\chi(u,w)$.
\end{lemma}
\begin{proof} We notice that the weak geodesic $[0,1] \ni t \to \alpha_t,\beta_t \in \mathcal H_\Delta$ connecting $u,v$ and $u,w$ respectively are both decreasing, satisfy $\alpha \geq \beta$ by the comparison principle, and $\alpha_0 =\beta_0$. From this it follows that $0 \geq \dot \alpha_0  \geq \dot \beta_0$. Using this, \eqref{ChiDistGeodFormula} yields the desired estimate.
\end{proof}

Our next result is the analog of \cite[Lemma 6.3]{d}:

\begin{lemma} Suppose $\{ u_k\}_{k \in \Bbb N} \subset \mathcal H$ is a sequence decreasing pointwise to $u \in \mathcal E_{\tilde \chi}(X,\o)$. Then $d_\chi(u_l,u_k) \to 0$ as $l,k \to \infty$.\label{IntDistEst}
\end{lemma}

\begin{proof} Our argument is just a small modification of the original proof. Suppose that $l \leq k$. Then $u_k\leq u_l$, hence by the previous result and Proposition \ref{NormIntegralEst} we have:
$$d_\chi(u_l,u_k) \leq \|u_l - u_k\|_{\chi,u_k} \leq M_{p}\Big(\int_X\chi(u_k-u_l)\o_{u_k}^n/\chi(1)\Big).$$
We clearly have $u - u_l, u_k - u_l \in \mathcal E_{\tilde \chi}(X,\o +i\partial\bar\partial u_l)$ and $u - u_l\leq u_k - u_l\leq0$. Hence, applying \cite[Lemma 3.5]{gz} for the class $\mathcal E_{\tilde \chi}(X,\o + i\partial\bar\partial u_l)$ we obtain that there exits $C=C(p)>0$ such that
\begin{equation}\label{estimate}
d_\chi(u_l,u_k)\leq M_{p}\Big(\int_X\chi(u_k-u_l)(\o_{u_k})^n/\chi(1)\Big) \leq CM_{p}\Big(\int_X\chi(u-u_l)\o_u^n\Big).
\end{equation}
As $u_l$ decreases to $u \in \mathcal E_{\tilde \chi}(X,\o)$, it follows from the dominated convergence theorem that $d_\chi(u_l,u_k) \to 0$ as $l,k \to \infty$.
\end{proof}

The next result is the analog of \cite[Lemma 6.4]{d} and its proof is the same as the original.

\begin{lemma} Given  $u_0,u_1 \in \mathcal E_{\tilde \chi}(X,\o)$, the limit in \eqref{EchiDistDef} is finite and independent of the approximating sequences $u^k_0, u^k_1 \in \mathcal H$.
\end{lemma}

To conclude that $d_\chi$ is a metric on $\mathcal E_{\tilde \chi}(X,\o)$ all we need is that $d_\chi(u_0,u_1)=0$ implies $u_0 = u_1$. This result is analogous to \cite[Lemma 6.7]{d} and using Proposition \ref{NormIntegralEst} its proof is carried out the same way:
\begin{lemma} Suppose $u_0,u_1 \in \mathcal E_{\tilde \chi}(X,\o)$  and $d_\chi(u_0,u_1)=0$. Then $u_0=u_1$.
\end{lemma}

By \cite[Theorem 6(i)]{d} it follows that given $u_0,u_1\in \mathcal E_{\tilde \chi}(X,\o)$, for the weak geodesic segment $t \to u_t$ connecting  $u_0,u_1$, as defined in \eqref{EchiGeodDef}, we have $u_t \in\mathcal E_{\tilde \chi}(X,\o), \ t \in (0,1)$. Next we show that this weak geodesic is an actual geodesic segment in $(\mathcal E_\chi(X,\o),d_\chi)$ in the sense of metric spaces. One needs a technical lemma generalizing \cite[Lemma 6.8]{d}:

\begin{lemma} \label{geod_tangent_limit} Suppose $v_0,v_1 \in \mathcal H_0=\textup{PSH}(X,\o)\cap L^\infty$ and $\{v^j_1 \}_{j \in \Bbb N}\subset \mathcal H_0$ is sequence decreasing to $v_1$. By $(0,1) \ni t \to v_t,v_t^j \in \mathcal H_0$ we denote the bounded weak geodesic segments connecting $v_0,v_1$ and $v_0, v^j_1$ respectively. As we have convexity in the $t-$variable, we can define $\dot v_0 = \lim_{t \to 0}(v_t - v_0)/t$ and $\dot v^j_0 = \lim_{t \to 0}(v^j_t - v_0)/t$. The following holds:
$$\lim_{j \to \infty}\|\dot {v_0^j}\|_{\chi,v_0} = \|\dot {v_0}\|_{\chi,v_0}.$$
\end{lemma}
\begin{proof} By an observation of Berndtsson (see Section 2.1 \cite{br}), there exists $C >0$ such that $\|\dot v_0\|_{L^\infty(X)}, \|\dot v^j_0\|_{L^\infty(X)} \leq C$. We also have $v \leq v^j, \ j \in \Bbb N$ by the comparison principle. As all $v^j$ are $t$--convex and share the same starting point, it also follows that $\dot v^j_0 \searrow \dot v_0$ pointwise. Proposition \ref{NormIntegralEst} and the dominated convergence theorem implies now that $\| \dot v_0^j - \dot v_0\|_{\chi,v_0} \to 0$.
\end{proof}
Given a metric space $(M,\rho)$, a curve $(0,1) \ni t \to h_t \in M$ is a (parametrized) geodesic, if there exists $c >0$ such that
$${\rho(h_l,h_s)}=c|l-s|, \ l,s \in (0,1).$$
Using the last lemma, the proof of the next result is carried out the same way the analagous result in \cite{d}.
\begin{lemma} Suppose $u_0,u_1 \in \mathcal E_{\tilde \chi}(X,\o)$ and $(0,1)\ni t \to u_t \in \mathcal E_{\tilde \chi}(X,\o)$ is the weak geodesic segment connecting $u_0,u_1$ defined in \eqref{EchiGeodDef}. Then $t \to u_t$ is a geodesic segment in $(\mathcal E_{\tilde \chi}(X,\o),d_\chi)$ in the sense of metric spaces.
\end{lemma}
Another technical lemma is needed:
\begin{lemma} Suppose $u,v \in \mathcal E_{\tilde \chi}(X,\o)$ satisfies $u \leq v$. Then:
$$d_\chi(u,v) \leq C M_p \Big(\int_X \chi(v-u)\o_u^n\Big)$$
where $C>0$ only depends on $\dim X$ and $p \geq 1$.
\end{lemma}
\begin{proof}Let $u_k,v_k \in \mathcal H$ be sequences decreasing to $u,v$, with the additional property $u_k \leq v_k, \ k \geq1$. By Lemma \ref{Mdist_est} and Remark \ref{NormIntegralEst} we have:
$$d_\chi(u_k,v_k) \leq \|v_k - u_k\|_{\chi,u_k} \leq C M_p \Big( \int_X \chi(u_k-v_k)\o_{u_k}^n\Big).$$
We clearly have $u-v_k, u_k -v_k \in \mathcal E_{\tilde \chi}(X,\o +i\partial\bar\partial v_k)$ and $u-v_k\leq u_k -v_k \leq 0$. Applying \cite[Lemma 3.5]{gz} to $\mathcal E_{\tilde \chi}(X,\o +i\partial\bar\partial v_k)$ we can conclude:
$$d_\chi(u_k,v_k) \leq C M_p \Big( \int_X \chi(u-v_k)\o_{u}^n\Big).$$
Letting $k \to \infty$, by the dominated convergence theorem we arrive at the desired estimate.
\end{proof}

Monotone sequences in $\mathcal E_{\tilde \chi}(X,\o)$ converge with respect to $d_\chi$. Using the last lemma, the proof of this result is the same as \cite[Proposition 6.11]{d}:

\begin{proposition} \label{Mdist_est_gen}
If $\{w_k\}_{k \in \Bbb N} \subset \mathcal E_{\tilde \chi}(X,\o)$ decreases (increases a.e.) to $w \in \mathcal E_{\tilde \chi}(X,\o)$ then $d_\chi(w_k,w)\to 0$.
\end{proposition}

As pointed out in \cite[Section 7]{d}, formula \eqref{ChiDistGeodFormula} fails already when $u_0,u_1$ are Lipschitz. Following \cite[Theorem 7.2]{d}, an extension is nevertheless possible for $u_0,u_1 \in \mathcal H_\Delta$. For this, first we have to adapt a theorem of Berndtsson \cite{br2} to our setting.

\begin{lemma} \label{pushforward_lemma}Suppose $u_0,u_1 \in \mathcal H_\Delta$ and $(0,1) \ni t \to u_t \in \mathcal H_\Delta$ is the geodesic connecting them, as guaranteed by \cite[Corollary 4.7]{bd} (see also \cite{h}). Then for any $f \in C(\Bbb R)$ the following holds:
\begin{equation}\label{lengthequal}
\int_X f(\dot u_0)\o_{u_0}^n=\int_X f(\dot u_t)\o_{u_t}^n, \ t \in [0,1].
\end{equation}
\end{lemma}
\begin{proof} We need to prove the desired identity only for $t=1.$ By translation we can assume that $f(0)=0$. Using approximation, we can also assume that $f \in C^1(\Bbb R)$. To obtain \eqref{lengthequal} for $t=1$ we have to prove the following two formulas:
\begin{equation}\label{lengthequal1}\int_{\{\dot u_0 >0\}} f(\dot u_0)\o_{u_0}^n=\int_{\{\dot u_1 > 0\}} f(\dot u_1)\o_{u_1}^n,
\end{equation}
\begin{equation}\label{lengthequal2} \int_{\{\dot u_0 <0\}} f(\dot u_0)\o_{u_0}^n=\int_{\{\dot u_1 < 0\}} f(\dot u_1)\o_{u_1}^n.
\end{equation}
As $\int_{\{\dot u_0 >0\}} f(\dot u_0)\o_{u_0}^n=\int_0^{\infty} f'(\tau)\o_{u_0}^n(\{{\dot u_0 \geq \tau}\})d\tau$, one can see that the same calculation that gave \cite[formulas (46),(47)]{d} (which deals with the particular case $f(l)=l^2$) also gives \eqref{lengthequal1} and \eqref{lengthequal2}.
\end{proof}

\begin{lemma} \label{distgeod_general} Suppose $u_0,u_1 \in \mathcal H_\Delta$ and $(0,1) \ni t \to u_t \in \mathcal H_\Delta$ is the geodesic connecting them. Then we have:
\begin{equation}\label{distgeod_formula}
d_\chi(u_0,u_1) = \|\dot u_t\|_{\chi,u_t}, \  t \in [0,1].
\end{equation}
\end{lemma}
\begin{proof} As usual, let $u^k_0,u^k_1 \in \mathcal H$ be sequences of potentials decreasing to $u_0,u_1$. Let $(0,1) \ni t \to u^{kl}_t \in \mathcal H_\Delta$ be the geodesic joining $u_0^k,u^l_1$. By (\ref{distgeod})
$$d_\chi(u_0^k,u^l_1) = \|\dot {u^{kl}_0}\|_{\chi,u_0^k}.$$
If we let $l \to \infty$, by Lemma \ref{geod_tangent_limit} and Proposition \ref{Mdist_est_gen} we obtain that
$$d_\chi(u_0^k,u_1) = \|\dot {u^{k}_0}\|_{\chi,u_0^k},$$
where $(0,1) \ni t \to u^k_t \in \mathcal H_\Delta$ is the geodesic connecting $u^k_0$ with $u_1$. Using the previous lemma we can write:
$$\chi(1)=\int_X \chi \Big( \frac{\dot u^k_0}{\| {\dot u^k_0}\|_{\chi, u^k_0}}\Big)\o_{u^k_0}^n = \int_X \chi \Big( \frac{\dot u^k_1}{\| {\dot u^k_0}\|_{\chi, u^k_0}}\Big)\o_{u^k_1}^n.$$
Hence,
$$d_\chi(u_0^k,u_1) =\|\dot {u^{k}_0}\|_{\chi,u_0} =\|\dot {u^{k}_1}\|_{\chi,u_1}.$$
Letting $k \to \infty$, another application of Lemma \ref{geod_tangent_limit} yields (\ref{distgeod_formula}) for $t=1$. The case $t=0$ follows by symmetry, and for $0 < t < 1$ the result follows because a subarc of a geodesic is again a geodesic.
\end{proof}

With this last lemma under our belt, we can prove that $P(\cdot,\cdot)$ is a contraction in both components with respect to $d_\chi$, which is the analog of \cite[Proposition 8.2]{d}:

\begin{proposition} \label{contractivity} Given $\chi \in \mathcal W^+_p$ and $u,v,w \in \mathcal E_{\tilde \chi}(X,\o)$ we have
$$d_\chi (P(u,v),P(u,w)) \leq d_\chi(v,w).$$
\end{proposition}

Before we can give the proof, we need to generalize the Pythagorean formula of \cite[Proposition 8.1]{d}.

\begin{proposition} \label{pythagorean} Given $u_0,u_1 \in \mathcal H_\Delta$, we have $P(u_0,u_1)\in \mathcal H_\Delta$. Let $(0,1) \ni t \to u_t,v_t,w_t \in \mathcal H_\Delta$ be the weak geodesics joining $(u_0,u_1)$, $(u_0,P(u_0,u_1))$ and $(u_1,P(u_0,u_1))$ respectively. Then we have
\begin{equation}\label{pythagorean_formula}
\int_X f(\dot u_t) \o_{u_t}^n = \int_X f(\dot v_l) \o_{v_l}^n + \int_X f(\dot w_h) \o_{w_h}^n, \ t,l,h \in [0,1],
\end{equation}
where $f \in C(\Bbb R)$ is arbitrary with $f(0)=0$.
\end{proposition}

\begin{proof} The fact that $P(u_0,u_1)\in \mathcal H_\Delta$ follows from \cite{dr}. Using an approximation argument we can assume that $f \in C^1(\Bbb R)$. To justify \eqref{pythagorean_formula}, by Lemma \ref{pushforward_lemma} it is enough to show that
$$\int_{\{\dot u_0 > 0\}} f(\dot u_0) \o_{u_0}^n = \int_X f(\dot w_0) \o_{u_1}^n.$$
$$\int_{\{\dot u_0 < 0 \}} f(\dot u_0) \o_{u_0}^n = \int_X f(\dot v_0) \o_{u_0}^n$$
As $\int_{\{\dot u_0 > 0 \}} f(\dot u_0) \o_{u_0}^n=\int_0^{\infty} f'(\tau)\o_{u_0}^n(\{{\dot u_0 \geq \tau}\})d\tau$, both of the above identities can be proved using the same arguments as in \cite[formulas (49) and (50)]{d}, where the particular case $f(l)=l^2$ is considered.
\end{proof}

\begin{corollary}For $u_0,u_1 \in \mathcal E^p(X,\o)$ we have
\begin{equation}\label{p-pythagorean_formula}
d_p(u_0,u_1)^p = d_p(u_0,P(u_0,u_1))^p + d_p(u_1,P(u_0,u_1))^p.
\end{equation}
Consequently,
\begin{equation}\label{1-pythagorean_formula}
d_1(u_0,u_1) = AM(u_0) + AM(u_1) - 2AM(P(u_0,u_1)).
\end{equation}
For general $\chi \in \mathcal W^+_p$ and $u_0,u_1 \in \mathcal E_{\tilde \chi}(X,\o)$ we have
\begin{flalign}\label{normdecompositionformula}
\frac{1}{2}(d_\chi(u_0,P(u_0,u_1)) + d_\chi(u_1, & P(u_0,u_1))) \leq d_\chi(u_0,u_1) \\
&\leq 2(d_\chi(u_0,P(u_0,u_1)) + d_\chi(u_1,P(u_0,u_1))). \nonumber
\end{flalign}
\end{corollary}
These results can also be seen as a generalization of the Pythagorean formula \cite[Proposition 8.1]{d} to our setting.
\begin{proof} To begin, we note that is enough to prove each indentity/estimate for $u_0,u_1 \in \mathcal H$, as in each case an approximation procedure yields the general result. To obtain \eqref{p-pythagorean_formula}, one just puts $f = \chi_p$ into the previous proposition.

Formula \eqref{1-pythagorean_formula} follows from the fact that $d_1(u_0,u_1)=AM(u_0) - AM(u_1)$ when $u_0 \geq u_1$. Indeed, the geodesic $t \to u_t$ connecting $u_0,u_1$ satisfies $\dot u_0 \leq 0$, hence $d_1(u_0,u_1)=\| \dot  u_0\|_{1,u_0} = - \int_X \dot u_0 \o_{u_0}^n=AM(u_0) - AM(u_1)$, as the Aubin-Mabuchi energy is linear along geodesics.

Trying to imitate all of this for general $\chi \in \mathcal W^+_p$ we choose $D = d_\chi(u_0,u_1)$ and $f(l)=\chi(l/D)$. From \eqref{pythagorean_formula} it follows that
$$\chi(1) = \int_X \chi \Big(\frac{\dot v_0}{D}\Big) \o_{v_0}^n + \int_X \chi\Big(\frac{\dot w_0}{D}\Big) \o_{w_0}^n,$$
with $t \to v_t$ and $t \to w_t$ as in the previous proposition.
From this identity and Lemma \ref{distgeod_general} it immediately follows that $d_\chi(u_0,P(u_0,u_1)),d_\chi(u_1,P(u_0,u_1)) \leq D$ implying the first estimate in \eqref{normdecompositionformula}. For the second estimate, we observe that at least one of the terms in the above identity is bigger then $\chi(1)/2$. We can assume that this term is the first one and convexity of $\chi$ implies
$$ \int_X \chi \Big(\frac{2\dot v_0}{D}\Big) \o_{v_0}^n\geq \int_X 2\chi \Big(\frac{\dot v_0}{D}\Big) \o_{v_0}^n\geq {\chi(1)},$$
from which it follows that $d_\chi(u_0,P(u_0,u_1)) \geq D/2$, implying the second estimate in \eqref{normdecompositionformula}.
\end{proof}

\begin{proof}[Proof of Proposition \ref{contractivity}]
As the general case follows from an approximation argument via Proposition \ref{Mdist_est_gen}, it is enough to prove the estimate for $u,v,w \in \mathcal H_\Delta$. In this particular case we also have $P(u,v),P(u,w), P(u,v,w) \in \mathcal H_\Delta$ (see \cite{dr}). We digress slightly to prove the following claim central to our argument:
\begin{claim*}Suppose $v,w \in \mathcal H_\Delta$, $v \leq w$ and let $(0,1) \ni t \to \phi_t,\psi_t \in \mathcal H_\Delta$ be the weak geodesic segments joining $v,w$ and $P(u,v),P(u,w)$ respectively. For any increasing non-negative function $f \in C([0,\infty))$ one has
\begin{equation}\label{interm_est}
\int_X f(\dot \psi_l)\o_{\psi_l}^n \leq \int_X f(\dot \phi_h)\o_{\phi_h}^n, \ l,h \in [0,1].
\end{equation}
\end{claim*}
\begin{proof}
As follows from Lemma \ref{pushforward_lemma} it is enough to prove the above estimate for $l=h =0$. From \eqref{MA_formula} it follows that that
$$\int_X f(\dot{\psi_0}) \o_{P(u,v)}^n\leq \int_{\{P(u,v)=u\}}f(\dot{\psi_0}) \o_{u}^n + \int_{\{P(u,v)=v\}}f(\dot{\psi_0}) \o_{v}^n$$
We argue that the first term in this sum is zero. As $v \leq w$ and $P(u,v) \leq P(u,w)$, it is clear that $t \to \phi_t,\psi_t$ are increasing in $t$. By the maximum principle, it is also clear that $\psi_t \leq u, \ t \in [0,1]$. Hence, if $x \in \{ P(u,v)=u\}$ then $\psi_t(x)=u(x), t \in [0,1]$, implying $\dot \psi_0 \big|_{\{ P(u,v)=u\}}\equiv 0$.

At the same time, using the maximum principle again, it follows that $\psi_t \leq \phi_t, \ t \in [0,1]$. This implies that $0 \leq \dot \psi_0 \big|_{\{ P(u,v)=v\}}\leq\dot \phi_0 \big|_{\{ P(u,v)=v\}},$ which in turn implies (\ref{interm_est}).
\end{proof}
The proof is just an application of Proposition \ref{pythagorean} and the claim just provided. Suppose $(0,1) \ni t \to a_t,b_t,c_t,d_t,e_t,g_t \in \mathcal H_\Delta$ are the weak geodesics joining $(P(u,v),P(u,w))$, $(P(u,v),P(u,v,w))$, $(P(u,w),P(u,v,w))$, $(v,P(v,w))$, $(w,P(v,w))$, and $(v,w)$ respectively. We introduce $D = \| \dot a_0\|_{\chi, P(u,v)}$, which by Lemma \ref{distgeod_general} is equals $d_\chi(P(u,v),P(u,w))$. Using Proposition \ref{pythagorean} we can start writing
\begin{flalign*}
\chi(1) = \int_X \chi\Big( \frac{\dot a_0}{D}\Big) \o_{P(u,v)}^n&= \int_X \chi\Big( \frac{\dot b_0}{D}\Big) \o_{P(u,v)}^n + \int_X \chi\Big( \frac{\dot c_0}{D}\Big) \o_{P(u,w)}^n.
\end{flalign*}
Now we use \eqref{interm_est} for both of the terms in the above sum to continue:
\begin{flalign*}
&\leq \int_X \chi\Big( \frac{\dot d_0}{D}\Big) \o_{v}^n + \int_X \chi\Big( \frac{\dot e_0}{D}\Big) \o_{w}^n\\
&= \int_X \chi\Big( \frac{\dot g_0}{D}\Big) \o_{v}^n,
\end{flalign*}
where in the last line we have used Proposition \ref{pythagorean} again. From the estimate we obtained it follows that $d_\chi(v,w)=\|\dot g_0\|_{v}\geq D = d_\chi(P(u,v),P(u,w))$.
\end{proof}

As an easy computation shows, for $u,v \in \mathcal E^1(X,\o)$ we have
\begin{equation}\label{AM_diff}AM(u)-AM(v)=\frac{1}{n+1} \sum_{j=0}^n\int_{X} (u-v)(\o +i\partial\bar\partial u)^j\wedge (\o + i\partial\bar\partial v)^{n-j}.
\end{equation}
We observe now that the Aubin-Mabuchi energy is Lipschitz continuous with respect to our path length metric:

\begin{lemma} \label{AMcont} Given $u_0,u_1 \in \mathcal E_{\tilde \chi}(X,\o)$, we have $$|AM(u_0) - AM(u_1)| \leq C(p)  d_\chi(u_0,u_1).$$
\end{lemma}
\begin{proof}
By density we can suppose that $u_0,u_1 \in \mathcal H.$ Let $(0,1) \ni t \to u_t \in \mathcal H_\Delta$ be the geodesic connecting $u_0,u_1$. By (\ref{ChiDistGeodFormula}), (\ref{AM_diff}) and the H\"older inequality we have:
\begin{flalign*}
AM(u_1) - AM(u_0)&= \int_0^1 \frac{d AM(u_t)}{dt}dt=\int_0^1 \int_X \dot {u_t}\o_{u_t}^n dt\\
&\leq \int_0^1\|1\|_{\chi^*,u_t} \|\dot u_t\|_{\chi,u_t} dt=\|1\|_{\chi^*,u_0} d_\chi(u_0,u_1),
\end{flalign*}
where we observed that all the norms $\|1\|_{\chi^*,u_t}$ are the same and are equal to $\|1\|_{\chi^*,u_0}$.
\end{proof}

We are ready to prove completeness of $(\mathcal E_\chi(X,\o), d_\chi)$. Roughly, the idea of the proof is to replace an arbitrary Cauchy sequence with an equivalent monotone Cauchy sequence which is much easier to deal with in light of the next result:
\begin{lemma}\label{mononton_seq} Suppose $\{u_k\}_{k \in \Bbb N} \subset \mathcal E_{\tilde \chi}(X,\o)$ is a pointwise decreasing $d_\chi-$bounded sequence. Then $u = \lim_{k \to \infty} u_k \in \mathcal E_{\tilde \chi}(X,\o)$ and additionally $d_\chi(u,u_k) \to 0$.
\end{lemma}
\begin{proof} We can suppose that $u_k < 0$. First we assume that $\{u_k\}_{k \in \Bbb N} \subset \mathcal H$. By Lemmas \ref{NormIntegralEst} and \ref{Mdist_est} it follows that
$$m_p\Big(\int_X {\chi}(u_j)\o_{u_j}^n / \chi(1)\Big) \leq \|u_j\|_{\chi,u_j} \leq 2^{n+1}d_\chi(u_j,0) \leq D.$$
Now \cite[Proposition 5.6]{gz} says that
$$\int_X {\tilde \chi}(u)\o_{u}^n < +\infty,$$
giving $u \in \mathcal E_{\tilde \chi}(X,\omega)$. The fact that $d_\chi(u_k,u) \to 0$ follows from Proposition \ref{Mdist_est_gen}. When elements of the sequence $\{u_k\}_{k \in \Bbb N}$ are non-smooth, the result follows from an approximation argument via Proposition \ref{Mdist_est_gen} again.
\end{proof}

The following theorem is the main result of this section. Its proof rests on Proposition \ref{contractivity} and the preceding two lemmas, and is carried out exactly the same way as the arguments in \cite[Theorem 9.2]{d}.
\begin{theorem} \label{EChiComplete} If $\chi \in \mathcal W^+_p, \ p \geq 1$ then $(\mathcal E_{\tilde \chi}(X, \o),d_\chi)$ is a geodesic metric space, which is the metric completion of $(\mathcal H,d_\chi)$. Additionally, the curve defined in \eqref{EchiGeodDef} is a geodesic connecting $u_0,u_1 \in \mathcal E_{\tilde \chi}(X, \o)$.
\end{theorem}

Unfortunately, geodesics connecting different points of $(\mathcal E_{\tilde \chi}(X, \o),d_\chi)$ may not be unique. This is most easily seen for $(\mathcal E^1(X, \o),d_1)$, as by \eqref{p-pythagorean_formula} we have
$$d_1(u_0,u_1)=d_1(u_0,P(u_0,u_1))+d_1(u_1,P(u_0,u_1)),$$
for all $u_0,u_1 \in \mathcal E^1(X,\omega)$. This formula implies that the geodesic connecting $u_0$ to $P(u_0,u_1)$ concatenated with the geodesic connecting $P(u_0,u_1)$ and $u_1$ has the same length as the geodesic connecting $u_0,u_1$.
The fact that geodesics connecting different points of $(\mathcal E_{\tilde \chi}(X, \o),d_\chi)$ may not be unique implies that $(\mathcal E_{\tilde \chi}(X, \o),d_\chi)$ is not a $CAT(0)$ space in general. This is in sharp contrast with the findings of \cite{d} about $(\mathcal E^2 (X, \o),d_2)$.

\section{Convergnce in $\chi$--Energy inside $\mathcal E_{\tilde \chi}(X,\o)$}

The first step in proving Theorem \ref{Energy_Metric_Eqv} is extending Lemma \ref{Mdist_est} for non-smooth potentials:

\begin{lemma} \label{Mdist_est_verygen}Suppose $\chi \in \mathcal W^+_p$ and $u_0,u_1 \in \mathcal E_{\tilde \chi}(X,\o)$ with $u_0 \leq u_1$. The following holds:
$$\max\{2^{-n-2} \|u_1-u_0\|_{\chi,u_0}, \|u_1-u_0\|_{\chi,u_1}\} \leq d_\chi(u_0,u_1) \leq \|u_1-u_0\|_{\chi,u_0}$$
\end{lemma}

The proof of this result follows using approximation with smooth potentials via Proposition \ref{Mdist_est_gen} and the next technical lemma:
\begin{lemma} Suppose $\chi \in \mathcal W^+_p$ and $\{u_k\}_{k \in \Bbb N},\{v_k\}_{k \in \Bbb N},\{w_k\}_{k \in \Bbb N} \subset \mathcal E_{\tilde \chi}(X,\o) $ are decreasing sequences for which $u_k \leq v_k, \ u_k \leq w_k$ and also $u_k \searrow u \in \mathcal E_{\tilde\chi}(X,\o)$, $v_k \searrow v \in \mathcal E_{\tilde\chi}(X,\o)$ and $w_k \searrow w \in \mathcal E_{\tilde\chi}(X,\o)$. Then we have
\begin{equation}
\label{auxlimit}
\lim_{k \to \infty} \|u_k - v_k\|_{\chi,w_k} = \|u - v\|_{\chi,w}
\end{equation}
\end{lemma}
\begin{proof} We can suppose without loss of generality that all the functions involved are negative and $N = \|u - v\|_{\chi,w} >0$. As $\chi \in \mathcal W^+_p$, by Proposition \ref{NormIntegralEst} it is enough to prove that
\begin{flalign}\label{boundedlimit}
\int_X \chi\Big( \frac{u_k - v_k}{N}\Big)\o_{w_k}^n -\chi(1)=\int_X \chi\Big( \frac{u_k - v_k}{N}\Big)\o_{w_k}^n - \int_X \chi\Big( \frac{u - v}{N}\Big)\o_{w}^n \to 0.
\end{flalign}
First we suppose that there exists $L >1$ such that $-L < u,u_k,v,v_k,w,w_k < 0$ are uniformly bounded. Given $\varepsilon >0$ one can find an open $O \subset X$ such that $\textup{Cap}_X(O) < \varepsilon$ and $u,u_k,v,v_k,w,w_k$ are all continuous on $X \setminus O$. We have
\begin{flalign*}
\int_X \chi\Big( \frac{u_k - v_k}{N}\Big)\o_{w_k}^n - \int_X \chi\Big( \frac{u - v}{N}\Big)\o_{w_k}^n=\int_O + \int_{X \setminus O} \Big[\chi\Big( \frac{u_k - v_k}{N}\Big) - \chi\Big( \frac{u - v}{N}\Big)\Big]\o_{w_k}^n.
\end{flalign*}
The integral on $O$ is bounded by $2\varepsilon L^n\chi(2L/N)$. The second integral tends to $0$ as on the closed set $X \setminus O$ we have $u_k \to u$ and $v_k \to v$ uniformly. We also have
\begin{flalign*}
\int_X \chi\Big( \frac{u - v}{N}\Big)\o_{w_k}^n - \int_X \chi\Big( \frac{u - v}{N}\Big)\o_{w}^n \to 0,
\end{flalign*}
as the function $\chi\Big( \frac{u - v}{N}\Big)$ is quasi--continuous and bounded \cite[Theorem 3.2]{bt}. This all means that \eqref{boundedlimit} follows when $-L < u,u_k,v,v_k,w,w_k <0$. Now we argue that this last limit also holds when $u,u_k,v,v_k,w,w_k$ are unbounded. For this we  show that
\begin{flalign}\label{uniformlimit}
\int_X \chi\Big( \frac{u_k - v_k}{N}\Big)\o_{w_k}^n - \int_X \chi\Big( \frac{u_k^L - v_k^L}{N}\Big)\o_{w_k^L}^n \to 0
\end{flalign}
\begin{flalign}\label{uniformlimit2}
\int_X \chi\Big( \frac{u - v}{N}\Big)\o_{w}^n - \int_X \chi\Big( \frac{u^L - v^L}{N}\Big)\o_{w^L}^n \to 0
\end{flalign}
as $L \to -\infty$ uniformly for $u,v,w,u_k,v_k,w_k$, where $h^L = \max (h,-L)$. Before we get into the estimates we observe that there exists $\psi \in \mathcal W^+_{2p+1}$ such that $\chi(l)/\psi(l)$ decreases to $0$ as $l \to -\infty$ and $u \in \mathcal E_\psi(X,\o)$. This implies that $u,v,w,u_k,v_k,w_k \in \mathcal E_\psi(X,\o), \ k \in \Bbb N$. As $\{ v_k \leq -L\},\{ w_k \leq -L\} \subseteq\{ u_k \leq -L\}$, since the Monge-Amp\'ere measure is local in the plurifine topology we can start writing:
\begin{flalign*}
\Big|\int_X \chi\Big( &\frac{u_k - v_k}{N}\Big)\o_{w_k}^n - \int_X \chi\Big( \frac{u_k^L - v_k^L}{N}\Big)\o_{w_k^L}^n \Big| = \\
&=\Big|\int_{\{u_k \leq -L\}} \chi\Big( \frac{u_k - v_k}{N}\Big)\o_{w_k}^n - \int_{\{u_k \leq -L\}} \chi\Big( \frac{u_k^L - v_k^L}{N}\Big)\o_{w_k^L}^n \Big| \\
&\leq\int_{\{u_k \leq -L\}} \chi\Big( \frac{u_k}{N}\Big)\o_{w_k}^n + \int_{\{u_k \leq -L\}} \chi\Big( \frac{u_k^L}{N}\Big)\o_{w_k^L}^n \\
&\leq\frac{\chi(L/N)}{\psi(L/N)}\Big(\int_{\{u_k \leq -L\}} \psi\Big( \frac{u_k}{N}\Big)\o_{w_k}^n + \int_{\{u_k \leq -L\}} \psi\Big( \frac{u_k^L}{N}\Big)\o_{w_k^L}^n \Big)\\
&\leq\frac{\chi(L/N)}{\psi(L/N)}\Big(\int_X \psi\Big( \frac{u_k}{N}\Big)\o_{w_k}^n + \int_X \psi\Big( \frac{u_k}{N}\Big)\o_{w_k^L}^n \Big)\\
&\leq\frac{C(p,N)\chi(L/N)}{\psi(L/N)}\Big(E_{\tilde \psi}(u_k) + E_{\tilde \psi}(w_k) + E_{\tilde \psi}(w_k^L) \Big)\\
&\leq\frac{C(p,N)\chi(L/N)}{\psi(L/N)} E_{\tilde \psi}(u),
\end{flalign*}
where in the penultimate line we have used \cite[Proposition 3.6]{gz} and  in the last line we have used \cite[Lemma 3.5]{gz}. This justifies \eqref{uniformlimit} and \eqref{uniformlimit2} is established the same way. As we explained above, all this implies that \eqref{boundedlimit} holds for arbitrary $u,u_k,v,v_k,w,w_k$.
\end{proof}
The following preliminary result will follow from \eqref{normdecompositionformula}, and is perhaps the crucial ingredient in proving Theorem \ref{Energy_Metric_Eqv}.
\begin{lemma} \label{halwayest} Suppose $\chi \in \mathcal W^+_p$ and $u_0,u_1 \in \mathcal E_{\tilde \chi}(X,\o)$. There exists $C(p) >1$ such that
$$d_\chi\Big(u_0,\frac{u_0+u_1}{2}\Big) \leq C d_\chi(u_0,u_1).$$
\end{lemma}

\begin{proof} Using \eqref{normdecompositionformula} multiple times we can start writing:
\begin{flalign*}
d_\chi\Big(u_0,\frac{u_0 + u_1}{2}\Big) &\leq C\Big[d_\chi\Big(u_0, P\Big(u_0,\frac{u_0 + u_1}{2}\Big)\Big) + d_\chi\Big(\frac{u_0 + u_1}{2},P\Big(u_0,\frac{u_0 + u_1}{2}\Big)\Big)\Big]\\
&\leq C\Big[d_\chi(u_0, P(u_0,u_1)) + d_\chi\Big(\frac{u_0 + u_1}{2},P(u_0,u_1)\Big)\Big]\\
&\leq C\Big[ \|u_0 - P(u_0,u_1)\|_{\chi,P(u_0,u_1)} + \Big\|\frac{u_0+u_1}{2} - P(u_0,u_1)\Big\|_{\chi,P(u_0,u_1)}\Big]\\
&\leq C\Big[ \frac{3}{2}\|u_0 - P(u_0,u_1)\|_{\chi,P(u_0,u_1)} + \frac{1}{2}\|u_1 - P(u_0,u_1)\|_{\chi,P(u_0,u_1)}\Big]\\
& \leq C\Big[d_\chi(u_0, P(u_0,u_1)) + d_\chi(u_1,P(u_0,u_1))\Big]\\
&\leq Cd_\chi(u_0,u_1),
\end{flalign*}
where in the second line we have used Lemma \ref{NaiveCompare} and the fact that $P(u_0,u_1) \leq P(u_0,(u_0 + u_1)/2)$, in the third and fifth line Lemma \ref{Mdist_est_verygen}, in the fourth line the subadditivity of the norm, and in the sixth line we have used \eqref{normdecompositionformula} again.
\end{proof}

The $\max$ operator interacts well with the $I_\chi$ energy. This is showcased in the next auxiliary lemma which is essentially a result of Guedj \cite[Proposition 2.16]{g} adapted to our more general setting:
\begin{lemma}\label{maxaddititvity}
Given $\chi \in \mathcal W^+_p$ we have
\begin{flalign*}
\frac{1}{2}(I_\chi(u_0,\max(u_0,u_1)) + &I_\chi(\max(u_0,u_1),u_1)) \leq I_\chi(u_0,u_1) \\
&\leq 2(I_\chi(u_0,\max(u_0,u_1)) + I_\chi(\max(u_0,u_1),u_1))
\end{flalign*}
for $u_0,u_1 \in \mathcal E_{\tilde \chi}(X,\o)$.
\end{lemma}
\begin{proof} Suppose $\alpha = \| u_1 - u_0\|_{\chi,u_0}$ and $\beta = \| u_1 - u_0\|_{\chi,u_1}$. As it will be clear at the end of the proof, we can assume that $\alpha,\beta \neq 0$. Then we have that
$$\int_X \chi\Big(\frac{u_1 - u_0}{\alpha}\Big)\o_{u_0}^n = \int_X \chi\Big(\frac{u_1 - u_0}{\beta}\Big)\o_{u_1}^n= \chi(1).$$
As the Monge-Amp\`ere operator is local in the plurifine topology we also have that
$$\int_X \chi\Big(\frac{\max(u_1,u_0) - u_0}{\beta}\Big)\o_{\max(u_0,u_1)}^n=\int_{\{ u_1 > u_0\}} \chi\Big(\frac{u_1 - u_0}{\beta}\Big)\o_{u_1}^n \leq \chi(1)$$
$$\int_X \chi\Big(\frac{\max(u_1,u_0) - u_0}{\alpha}\Big)\o_{u_0}^n=\int_{\{ u_1 > u_0\}} \chi\Big(\frac{u_1 - u_0}{\alpha}\Big)\o_{u_0}^n \leq \chi(1)$$
This implies that $I_\chi(u_0,\max(u_0,u_1)) \leq \alpha +\beta = I_\chi(u_0,u_1)$ and by symmetry we also have $I_\chi(u_1,\max(u_0,u_1)) \leq \alpha +\beta = I_\chi(u_0,u_1)$. Adding these two estimates yields the first inequality of the lemma.
For the second inequality we first notice that
$$\int_X \chi\Big(\frac{\max(u_1,u_0) - u_0}{\alpha}\Big)\o_{u_0}^n + \int_X \chi\Big(\frac{\max(u_1,u_0) - u_1}{\alpha}\Big)\o_{\max(u_0,u_1)}^n =\chi(1)$$
Hence one of terms in the above sum is bigger then $\chi(1)/2$. We can assume that this is the first term, hence by convexity of $\chi$ we have
$$\int_X \chi\Big(\frac{2(\max(u_1,u_0) - u_0)}{\alpha}\Big)\o_{u_0}^n \geq\int_X 2\chi\Big(\frac{\max(u_1,u_0) - u_0}{\alpha}\Big)\o_{u_0}^n \geq {\chi(1)}$$
This implies that $I_\chi(\max(u_1,u_0) , u_0) \geq \| \max(u_1,u_0) - u_0\|_{\chi,u_0} \geq \alpha/2$. One can establish a similar estimate for $\beta$ and this finishes the proof.
\end{proof}

We can now establish the main theorem of this section:
\begin{theorem}\label{Energy_Metric_Eqv_thm}
Given $\chi \in \mathcal W^+_p$, there exists $C(p) > 1$ such that
$$\frac{1}{C}I_\chi(u_0-u_1) \leq d_{\chi}(u_0,u_1) \leq C I_\chi(u_0 - u_1), \ u_0,u_1 \in \mathcal E_{\tilde \chi}(X,\o).$$
\end{theorem}
\begin{proof} The second estimate follows easily:
\begin{flalign*}
d_{\chi}(u_0,u_1) &\leq
d_{ \chi}(u_0,\max(u_0,u_1)) + d_{ \chi}(\max(u_0,u_1),u_1)\\
&\leq I_\chi(u_0,\max(u_0,u_1)) + I_\chi(\max(u_0,u_1),u_1)\\
&\leq C I_\chi(u_0,u_1),
\end{flalign*}
where in the second line we have used Lemma \ref{Mdist_est_verygen} and in the third Lemma \ref{maxaddititvity}. Now we deal with the first estimate. By the previous result, \eqref{normdecompositionformula} and Lemma \ref{Mdist_est_verygen} we can write
\begin{flalign*}
d_\chi(u_0,u_1) &\geq C d_\chi\Big(u_0,\frac{u_0 + u_1}{2}\Big) \geq Cd_\chi\Big(u_0,P\Big(u_0,\frac{u_0 + u_1}{2}\Big)\Big)\\
&\geq C\Big\|u_0 - P\Big(u_0,\frac{u_0 + u_1}{2}\Big)\Big\|_{\chi,u_0}.
\end{flalign*}
By the same reasoning as above and the fact that $2^n \o^n_{(u_0 + u_1)/2} \geq \o^n_{u_0}$ we can write:
\begin{flalign*}
d_\chi(u_0,u_1) &\geq C d_\chi\Big(u_0,\frac{u_0 + u_1}{2}\Big) \geq C d_\chi\Big(\frac{u_0+u_1}{2},P\Big(u_0\frac{u_0 + u_1}{2}\Big)\Big)\\
&\geq C\Big\|\frac{u_0+u_1}{2} - P\Big(u_0\frac{u_0 + u_1}{2}\Big)\Big\|_{\chi,(u_0 + u_1)/2}\\
&\geq C\Big\|\frac{u_0+u_1}{2} - P\Big(u_0\frac{u_0 + u_1}{2}\Big)\Big\|_{\chi,u_0}.
\end{flalign*}
Using the triangle inequality and the last two estimates, we obtain:
\begin{flalign*}
d_\chi(u_0,u_1)&\geq C\Big[\Big\|u_0 - P\Big(u_0\frac{u_0 + u_1}{2}\Big)\Big\|_{\chi,u_0} + \Big\|\frac{u_0+u_1}{2} - P\Big(u_0\frac{u_0 + u_1}{2}\Big)\Big\|_{\chi,u_0}\Big]\\
&\geq C\|u_0 - u_1\|_{\chi,u_0}.
\end{flalign*}
By symmetry we also obtain $d_\chi(u_0,u_1) \geq C\|u_0 - u_1\|_{\chi,u_1}$ and adding these last two estimates together the result follows.
\end{proof}

\begin{remark} \label{maxdistremark} From the previous two results it follows that there exists $C(p) >1$ for which
\begin{flalign*}
\frac{1}{C} (d_\chi(u_0, \max(u_0,u_1)) + &d_\chi(u_1, \max(u_0,u_1))) \leq d_\chi(u_0,u_1) \\
&\leq C(d_\chi(u_0, \max(u_0,u_1)) + d_\chi(u_1, \max(u_0,u_1))),
\end{flalign*}
for any $u_0,u_1 \in \mathcal E_{\tilde \chi}(X,\o)$. These estimates are the "$\max$" analog of the ones in \eqref{normdecompositionformula}.
\end{remark}

For $u_0,u_1 \in \mathcal E^1(X,\o)$, as proposed in \cite{bbgz}, one can introduce another notion of energy:
\begin{equation}\label{mathcalIdef}
\mathcal I(u_0,u_1) = \int_X (u_0 - u_1)(\o_{u_1}^n - \o_{u_0}^n).
\end{equation}
This functional will be useful for us in the proof of the next theorem.
\begin{corollary}\label{MabConvCor} Suppose $\chi \in \mathcal W^+_p$ and $d_\chi(u_k,u) \to 0$ with $\ u_k,u \in \mathcal E_{\tilde \chi}(X,\o)$. Then $u_k \to u$ in capacity in the sense of \cite{k}. In particular, $\o_{u_k}^n \to \o_{u}^n$ weakly.

Also, for any $v \in \mathcal E_1(X,\o)$ and $R > 0$ there exists a continuous increasing function $f_R:\Bbb R_+ \to \Bbb R_+$ with $f(0)=0$ such that
\begin{equation}\label{L1MabEst}
\int_X |u_0-u_1| \o_v^n \leq f(d_1(u_0,u_1)),
\end{equation}
for $u_0,u_1 \in \mathcal E_1(X,\o)$ with $d_1(0,u_0),d_1(0,u_1) \leq R$.
\end{corollary}
\begin{proof}As $d_1$ is dominated by all $d_\chi$ norms, we will only establish the result for $d_1$. We start with \eqref{L1MabEst}. From \cite[Lemma 5.8]{bbgz} and its proof it follows that for any $R >0$ there exists an increasing continuous function $f_R:\Bbb R_+ \to \Bbb R_+$ with $f(0)=0$ such that
\begin{equation}\label{BBGZLemmaestimate}
\int_X (u_0 - u_1)(\o_{v_1}^n - \o_{v_2}^n) \leq f(\mathcal I(u_0,u_1))
\end{equation}
for any $u_0, u_1, v_1,v_2$ having supremum less then $R$ and Aubin-Mabuchi energy greather then $-R$. By and Corollary \ref{supcorollary} and Lemma \ref{AMcont}, $d_1(0,u_0)$ dominates both $\sup_X u_0$ and $AM(u_0)$. Also by Theorem \ref{Energy_Metric_Eqv_thm}, $d_1(u_0,u_1)$ dominates $\mathcal I(u_0,u_1)$ and $\int_X (u_0-u_1)\o_{u_0}^n$. Hence, choosing $v_1 =v$ and $v_2 = u_0$, we can rewrite \eqref{BBGZLemmaestimate} in the following way
\begin{equation}\label{L1estimate}
\int_X (u_0 - u_1)\o_{v}^n \leq \tilde f(d_{1}(u_0,u_1)),
\end{equation}
for any $u_0,u_1 \in \mathcal E_1(X,\o)$ satisfying $d_1(0,u_0),d_1(0,u_1) \leq R$.
Using this, the identity
$$\int_X |u_0 - u_1|\o_v^n = 2 \int_X (\max(u_0,u_1) - u_1)\o_v^n -  \int_X (u_0 - u_1)\o_v^n,$$
and Remark \ref{maxdistremark}, estimate \eqref{L1MabEst} follows.

Now we prove that $d_\chi(u_k,u) \to 0$ implies $u_k \to u$ in capacity. This again follows from the results of \cite{bbgz}. Indeed, by \eqref{L1MabEst}, it follows that
$$\int_X u_k \o^n \to \int_X u \o^n.$$
Now using \cite[Theorem 5.7]{bbgz} it follows that $u_k \to u$ in capacity. Finally, $u_k \to u$ in capacity (for  $u_k,u \in \mathcal E(X,\o)$) is known to imply the weak convergence of measures $\o_{u_k}^n \to \o_{u}^n$ (see \cite{dp}, compare \cite[Proposition 5.6]{bbgz}).
\end{proof}

Building on this last result, we prove another continuity theorem:

\begin{theorem} \label{FlatImmersion}Suppose $\chi \in \mathcal W^+_p$ and $u_k, u, v \in \mathcal E_{\tilde \chi}(X,\o)$ with $d_\chi(u_k,u) \to 0$. Then $u_k,u \in L^\chi(\o_v^n)$ and
\begin{equation}\label{LchiMabuchi}
\| u_k - u\|_{\chi,v} \to 0.
\end{equation}
\end{theorem}
\begin{proof} The fact that $u_k,u \in L^\chi(\o_v^n)$ follows from \cite[Proposition 3.6]{gz}. By Proposition \ref{NormIntegralEst} to prove \eqref{LchiMabuchi}, it suffices to argue that
$$\int_X \chi(u_k - u) \o_v^n \to 0.$$
Given an arbitrary subsequence of $u_k$, there exists a sub-subsequence, again denoted by $u_k$, satisfying the sparsity condition:
\begin{equation}\label{sparsity}
d_\chi(u_k,u_{k+1}) \leq \frac{1}{2^k}, \ \ k \in \Bbb N.
\end{equation}
Using the sparsity condition and Proposition \ref{contractivity} we can write
\begin{flalign*}
d_\chi(P(u,u_0,\ldots,u_k),P(u,&u_0,\ldots,u_{k+1}))= \\ &=d_\chi(P(P(u,u_0,\ldots,u_k),u_k),P(P(u,u_0,\ldots,u_k),u_{k+1}))\\
&\leq d_\chi(u_k,u_{k+1}) \leq \frac{1}{2^k}.
\end{flalign*}
Hence, the decreasing sequence $h_k = P(u,u_0,u_1, \ldots,u_k), \ k \geq 1$ is bounded and Lemma \ref{mononton_seq} implies now that the decreasing limit $\lim_k h_k= h = P(u,u_0,u_1,u_2, \ldots) \in \mathcal E_{\tilde \chi}(X,\o).$

By Corollary \ref{supcorollary}, there exists $M >0$ such that
$h \leq u_k,u \leq M$. Putting everything together we have
$$h - M \leq u_k - u \leq M - h.$$
This implies that $\int_X \chi(u_k - u) \o_v^n  \leq \int_X \chi(M - h) \o_v^n < \infty.$ As $d_\chi(u_k,u) \to 0$ implies $u_k \to u$ in capacity, which in turn implies pointwise convergence a.e., by the dominated convergence theorem it follows that $\int_X \chi(u_k - u) \o_v^n \to 0$, finishing the proof.
\end{proof}
As promised we argue that the strong convergence introduced in \cite[Section 2.1]{bbegz} is equivalent to $\chi_1$-convergence. This results is well known to experts, however there does not seem to exist an adequate reference for it in the literature:

\begin{proposition}\label{strongchiconveqv} Suppose $u_k,u \in \mathcal E^1(X,\o)$. Then $\int_X |u_k - u|\o^n \to 0$ and $AM(u_k) \to AM(u)$ if and only if $I_1(u_k,u) \to 0$.
\end{proposition}
\begin{proof} As follows from Lemma \ref{AMcont} and Corollary \ref{MabConvCor}, $I_1(u_k,u) \to 0$ implies $\int_X |u_k - u|\o^n \to 0$ and $AM(u_k) \to AM(u)$. Now we argue the other direction. By \cite[Propostion 2.3]{bbegz} it follows that $\mathcal I(u_k,u) \to 0$. As in the proof of Corollary \ref{MabConvCor}, we can conclude that $\int_X (u_k - u)\o^n_{u_k} \to 0$ and $\int_X (u_k - u)\o^n \to 0$. Using the locality of the complex Monge-Amp\'ere measure in the plurifine topology we observe
$$\mathcal I(\max(u_k,u),u)= \int_{\{u_k > u\}} (u_k - u)(\o^n_{u} - \o^n_{u_k}) \leq \mathcal I(u_k,u),$$
hence we also have $\mathcal I(\max(u_k,u),u) \to 0$. By similar considerations as above we get $\int_X (\max(u_k,u) - u)\o^n_{u_k} \to 0$ and $\int_X (\max(u_k,u) - u)\o^n \to 0$.
This concludes the argument as we note that $|u_k - u|= 2(\max(u_k,u) - u)-(u_k - u).$
\end{proof}

\section{Applications to K\"ahler-Einstein metrics}

For this section we assume that $(X,J,\o)$ is Fano, i.e. $c_1(X)=[\o]$. We review a few facts about convergence of the K\"ahler--Ricci flow. Under the normalization $\int_X e^{-\dot r_t}\o^n = \textup{Vol}(X),$
equation \eqref{RicciflowEq} can be rewritten as the scalar equation
$$e^{\dot r_t - r_t  - h - \log \left({\int_X e^{-r_t - h}\o^n}\right) } \o^n= {\o_{r_t}^n}.$$
We refer to \cite[Section 6.2]{beg} for details. As it follows from a theorem of Perelman and work of Chen-Tian, Tian-Zhu and Phong-Song-Sturm-Weinkove, $\o_{r_t}$ converges exponentially fast to some K\"ahler-Einstein metric $\o_{u_{KE}}$, whenever such metric exists (see \cite{ct}, \cite{tz}, \cite{pssw}). More precisely we have $\| {d \o_{r_t}}/{dt}\|_{\o}=\| i\del\dbar \dot r_t\|_{\o} \leq C e^{-Ct}$ for some $C >0$. Given our normalization assumption, this implies that $\|\dot r_t\|_{L^\infty(X)} \leq C e^{-Ct}$, in particular $\|r_t - u_{KE}\|_{L^\infty(X)} \to 0.$
This will be the starting point in establishing the equivalence between (i) and (ii) in Theorem \ref{ApplKE}:
\begin{theorem}\label{MR-stab} There exists a K\"ahler-Einstein metric cohomologous to $\o$ if and only if $(X,J,\o)$ is $\chi$--stable for any $\chi \in \mathcal W^+_p, \ p \geq 1$.
\end{theorem}

\begin{proof} Suppose that there exists K\"ahler-Einstein metric cohomologous to $\o$ and $t \to \tilde r_t$ is a Ricci trajectory normalized by $\int_X e^{-d{{\tilde r}_t/dt}}\o^n = \textup{Vol}(X)$ with initial metric $v \in \mathcal H_{AM}$. By what we recalled above, $\tilde r_t$ converges to some K\"ahler-Einstein potential $\tilde u_{KE} \in \mathcal H$ uniformly.

As $AM(\cdot)$ is continuous with respect to uniform convergence,  it follows that the "AM--normalized" K\"ahler--Ricci trajectory $t \to r_t = \tilde r_t - AM(\tilde r_t) \in \mathcal H_{AM}$ also converges to $u_{KE} = \tilde u_{KE} - AM(\tilde u_{KE})$ uniformly, hence also with respect to $d_\chi$, finishing one direction of the proof.

For the other direction, let us assume that there exists no K\"ahler-Einstein metric cohomologous to $\o$. By \cite[Lemma 6.5]{cr}, along any K\"ahler--Ricci trajectory $[0,\infty) \ni t \to r_t \in \mathcal H_{AM}$, one can find a sequence $t_j \to \infty$ and an non-empty subvariety $S \subset X$ of positive dimension such that
$$\lim_{j \to \infty}\int_K \o_{r_{t_j}}^n=0,$$
for any $K \subset X \setminus S$ compact.
This implies that there is no $u \in \mathcal E_{\tilde \chi}(X,\o)$ satisfying $d_\chi(r_{t_j},u) \to 0$, as by Corollary \ref{MabConvCor} this would imply that $\o_{r_{t_j}}^n \to \o_{u}^n$ weakly. By the above, this in turn implies that $\o_{u}^n$ is concentrated on $S$, in particular $\int_S \o_u^n = \textup{Vol}(X) >0$. But this is impossible as $\o_{u}^n$ does not charge pluripolar sets, as follows from the definition of $\o_{u}^n$ (see \cite[Theorem 1.3]{gz}).
\end{proof}
Now we prove the equivalence between (i) and (iii) in Theorem \ref{ApplKE}:
\begin{theorem} \label{Dingstab}Suppose that $(X,J)$ does not have non-trivial holomorphic vector fields. Then there exists a K\"ahler-Einstein metric cohomologous to $\o$ if and only if $\mathcal F$ is $d_1$--proper on $\mathcal E^1_{AM}(X,\o)$.
\end{theorem}

\begin{proof} By \cite[Theorem 1.6]{t} it is enough to prove that the concept of $\mathcal J$--properness is equivalent to $d_1$--properness. Suppose $u \in \mathcal E^1_{AM}(X,\o)$. By Theorem \ref{Energy_Metric_Eqv_thm} there exists $C >1$ such that
\begin{equation}\label{proprness1}
\int_X u\o^n \leq C d_1(u,0).
\end{equation}
On the other hand, by \eqref{1-pythagorean_formula} and the fact that $AM(u)=AM(0)=0$ we can write:
$$d_1(u,0)=AM(u) + AM(0) - 2AM(P(0,u))=2(AM(u)- AM(P(0,u))).$$
We have $\sup_X (u - P(0,u)) \leq \sup_X u,$ hence \eqref{AM_diff} implies that $AM(u)- AM(P(0,u)) \leq \sup_X u$. It is well know that there exists $C>0$ depending only on $X$ and $\o$ such that $\sup_X u-C \leq \int_X u\o^n \leq \sup_X u$ for all $u \in \textup{PSH}(X,\o)$. Putting the above together we arrive at:
\begin{equation}\label{proprness2}
d_1(u,0) \leq 2 \int_X u\o^n + 2C.
\end{equation}
By \eqref{proprness1} and \eqref{proprness2} we have
\begin{equation}\label{Jpropd1prop}
\mathcal J(u)/C \leq d_1(u,0) \leq 2 \mathcal J(u) + 2C,
\end{equation}
hence $\mathcal J$--properness is equivalent to $d_1$--properness.
\end{proof}

From \eqref{Jpropd1prop} the following estimate of independent interest follows:
\begin{remark} \label{JensenRemark} There exists $C >1$ such that
$$\frac{1}{C}\sup_X u - C \leq d_1(u,0) \leq C \sup_X u + C, \ u \in \mathcal E^1_{AM}(X,\o).$$
As a consequence of this, there exists $C>1$ such that $\mathcal F(u) \leq C d_1(0,u) +C$ for any $u \in \mathcal E^1_{AM}(X,\o)$, regardless weather $(X,J,\o)$ admits a K\"ahler-Einstein metric or not.
\end{remark}

\vspace{0.1 in}
\textsc{Purdue University and University of Maryland}\\
\emph{E-mail address: }\texttt{\textbf{tdarvas@math.umd.edu}}

\begin{thebibliography}{1}
\bibitem[BT]{bt} E. Bedford, B.A. Taylor, Fine topology, Silov boundary, and $(dd\sp c)\sp n$. J. Funct. Anal. 72 (1987), no. 2, 225--251.
\bibitem[Brm1]{brm1} R. Berman, K-polystability of Q-Fano varieties admitting Kahler-Einstein metrics, arXiv:1205.6214.
\bibitem[Brm2]{bpers} R. Berman, personal communication, February 2014.
\bibitem[BrmBrn1]{bb} R. Berman, R. Berndtsson,  Convexity of the K-energy on the space of K\"ahler metrics, arXiv:1405.0401.
\bibitem[BrmBrn2]{bb2}  Convexity of the K-energy, entropy and the finite energy Calabi flow. Preprint 2014.
\bibitem[BBGZ]{bbgz} R. Berman, S. Boucksom, V. Guedj, A. Zeriahi, A variational approach to complex Monge-Amp\`ere equations,  Publ. Math. Inst. Hautes Études Sci. 117 (2013), 179--245.
\bibitem[BD]{bd} R. Berman, J. P. Demailly, Regularity of plurisubharmonic upper envelopes in big cohomology classes. Perspectives in analysis, geometry and topology, Progr. Math. 296, Birkhäuser/Springer, New York, 2012, 39--66.
\bibitem[Brn1]{br} B. Berndtsson, A Brunn-Minkowski type inequality for Fano manifolds and the Bando-Mabuchi uniqueness theorem,	 arXiv:1103.0923.
\bibitem[Brn2]{br2} B. Berndtsson, Probability measures related to geodesics in the space of K\"ahler metrics, arXiv:0907.1806.
\bibitem[Bl1]{bl2} Z. Blocki, On geodesics in the space of Kähler metrics, Proceedings of the "Conference in Geometry" dedicated to Shing-Tung Yau (Warsaw, April 2009), in "Advances in Geometric Analysis", ed. S. Janeczko, J. Li, D. Phong, Advanced Lectures in Mathematics 21, pp. 3-20, International Press, 2012.
\bibitem[Bl2]{bl1}Z. B\l ocki, The complex Monge-Amp\`ere equation in K\"ahler geometry, CIME Summer School in Pluripotential Theory, Cetraro, July 2011, to appear in Lecture Notes in Mathematics.
\bibitem[BK]{bk} Z. B\l ocki, S.Ko\l odziej, On regularization of plurisubharmonic functions on manifolds, Proceedings of the American Mathematical Society 135 (2007), 2089--2093.
\bibitem[BBEGZ]{bbegz} S. Boucksom, R. Berman, P. Eyssidieux, V. Guedj, A. Zeriahi, K\"ahler-Einstein metrics and the K\"ahler-Ricci flow on log Fano varieties, arXiv:1111.7158.
\bibitem[BEGZ]{begz} S. Boucksom, P. Eyssidieux, V. Guedj and A. Zeriahi, Monge-Amp\`ere equations in big cohomology classes, Acta Math. 205 (2010), 199--262.
\bibitem[Clb]{cal} E. Calabi, The variation of K\"ahler metrics, Bull. Amer. Math. Soc. 60 (1954), 167-168.
\bibitem[Clm]{clm} S. Calamai, The Calabi metric for the space of K\"ahler metrics. Math. Ann. 353 (2012), no. 2, 373-402.
\bibitem[Cao]{ca} H. Cao, Deformation of Kaehler metrics to Kaehler-Einstein metrics on compact Kähler manifolds, Invent. Math. 81 (1985), no. 2, 359--372.
\bibitem[C1]{c}X.X. Chen, The space of K\"ahler metrics, J. Differential Geom. 56 (2000), no. 2, 189--234.
\bibitem[CDS1]{cds1} X.X. Chen, S.K. Donaldson, S. Sun, K\"ahler-Einstein metrics and stability, arXiv:1210.7494.
\bibitem[CDS2]{cds2} X.X. Chen, S.K. Donaldson, S. Sun, K\"ahler-Einstein metrics on Fano manifolds, II: limits with cone angle less than $2\pi$, arXiv:1212.4714.
\bibitem[CDS3]{cds3} X.X. Chen, S.K. Donaldson, S. Sun, K\"ahler-Einstein metrics on Fano manifolds, III: limits as cone angle approaches $2\pi$ and completion of the main proof, arXiv:1302.0282.
\bibitem[CT]{ct} X.X Chen, G. Tian, Ricci flow on K\"ahler-Einstein surfaces, Inventiones mathematicae 2/2002; 147(3):487-544.
\bibitem[CR]{cr}B. Clarke, Y.A. Rubinstein, Ricci flow and the metric completion of the space of Kähler metrics, American Journal of Mathematics 135 (2013), no. 6, 1477-1505.
\bibitem[Da1]{d} T. Darvas, Envelopes and Geodesics in Spaces of K\"ahler Potentials, { arXiv:1401.7318}.
\bibitem[Da2]{da2} T. Darvas, Weak Geodesic Rays in the Space of Kähler Metrics and the Class $\mathcal E(X,\o_0)$, arXiv:1307.6822.
\bibitem[Da3]{da3} T. Darvas, Morse theory and geodesics in the space of Kähler metrics, arXiv:1207.4465v3.
\bibitem[DL]{dl} T. Darvas, L. Lempert, Weak geodesics in the space of K\"ahler metrics, 	 Mathematical Research Letters, 19 (2012), no. 5.
\bibitem[DR]{dr} T. Darvas, Y.A. Rubinstein, Kiselman's principle, the Dirichlet problem for the Monge-Amp\`ere equation, and rooftop obstacle problems, arXiv:1405.6548.
\bibitem[DP]{dp} S. Dinew, P.H. Hiep, Convergence in capacity on compact K\"ahler manifolds, Ann. Sc. Norm. Super. Pisa Cl. Sci. (5) 11 (2012), no. 4, 903--919.
\bibitem[Do]{do} S.K. Donaldson - Symmetric spaces, K\"ahler geometry and Hamiltonian dynamics, Amer. Math. Soc. Transl. Ser. 2, vol. 196, Amer. Math. Soc., Providence RI, 1999, 13--33.
\bibitem[EGZ]{egz} P. Eyssidieux, V. Guedj, A. Zeriahi, Singular K\"ahler-Einstein metrics. Journal of the AMS 22 (2009), 607--639.
\bibitem[G]{g} V. Guedj, The metric completion of the Riemannian space of Kähler metrics, arXiv:1401.7857.
\bibitem[BEG]{beg} S. Boucksom, P. Eyssidieux, V. Guedj,
An introduction to the Kähler-Ricci flow,  Lecture Notes in Math. 2086, Springer 2013.
\bibitem[GZ1]{gz} V. Guedj, A. Zeriahi, The weighted Monge-Amp\`ere energy of quasiplurisubharmonic functions, 	J. Funct. Anal. 250 (2007), no. 2, 442--482.
\bibitem[H1]{h} W. He, On the space of K\"ahler potentials, arXiv:1208.1021.
\bibitem[H2]{h2} W. He, $\mathcal F$-functional and geodesic stability, arXiv:1208.1020.
\bibitem[K]{k} S. Ko\l odziej, The Monge-Amp\`ere equation on compact K\"ahler manifolds, Indiana Univ. Math. J. 52 (2003), 667--686.
\bibitem[LV]{lv}L. Lempert, L. Vivas, Geodesics in the space of Kähler metrics. Duke Math. J. 162 (2013), no. 7, 1369--1381.
\bibitem[Ma]{m}T. Mabuchi, Some symplectic geometry on compact K\"ahler manifolds I, Osaka J. Math. 24, 1987, 227--252.
\bibitem[Mc]{mc} D. McFeron, The Mabuchi metric and the Kähler-Ricci flow, Proc. Amer. Math. Soc. 142 (2014), no. 3, 1005--1012.
\bibitem[PSSW]{pssw} D.H. Phong, J. Song, J. Sturm, B. Weinkove, The K\"ahler-Ricci flow and the $\bar \del$--operator on vector fields, J. Differential Geom., 81 (2009), no. 7, 631--647.
\bibitem[R]{r} Y. Rubinstein, Smooth and singular Kahler-Einstein metrics, arXiv:1404.7451.
\bibitem[RR]{rr} M.M. Rao, Z.D. Ren, Theory of Orlicz spaces. Monographs and Textbooks in Pure and Applied Mathematics, 146. Marcel Dekker, Inc., New York, 1991.
\bibitem[Se]{s}S. Semmes, Complex Monge-Amp\`ere and symplectic manifolds, Amer. J. Math. 114 (1992), 495--550.
\bibitem[St1]{st1} J. Streets, Long time existence of Minimizing Movement solutions of Calabi flow, arXiv:1208.2718.
\bibitem[St2]{st2} J. Streets, The consistency and convergence of K-energy minimizing movements, arXiv:1301.3948.
\bibitem[T1]{t}G. Tian, K\"ahler-Einstein metrics with positive scalar curvature. Invent. Math. 130 (1997), no. 1, 1--37.
\bibitem[T2]{t2} G. Tian, K-stability and K\"ahler-Einstein metrics, arXiv:1211.4669.
\bibitem[TZ]{tz} G. Tian, X. Zhu, Convergence of K\"ahler-Ricci flow. J. Amer. Math. Soc. 20 (2007), no. 3, 675--699.
\end{thebibliography}
\end{document}